\documentclass[reqno]{amsart}
\usepackage{amsmath}
\usepackage{amssymb}
\usepackage{amsthm}
\usepackage{graphicx}
\usepackage{color}

\newtheorem{theorem}{Theorem}[section]
\newtheorem{proposition}[theorem]{Proposition}
\newtheorem{corollary}[theorem]{Corollary}
\newtheorem{lemma}[theorem]{Lemma}
\theoremstyle{definition}
\newtheorem{definition}[theorem]{Definition}

\newtheorem{question}[theorem]{Question}
\newtheorem{remark}[theorem]{Remark}

\newcommand{\oline}[1]{\mathbin{\overline{#1}}}

\newcommand{\ostar}{\mathbin{\overline{\star}}}
\newcommand{\ustar}{\mathbin{\underline{\star}}}

\newcommand{\mincol}{{\rm mincol}^{\rm Dehn}}
\newcommand{\mincolF}{{\rm mincol}^{\rm Fox}}

\newcommand{\R}{\mathcal{R}}
\newcommand{\C}{\mathcal{C}}

\numberwithin{equation}{section}

\begin{document}

\title[]{Minimum numbers of Dehn colors of knots and symmetric local biquandle cocycle invariants}

\author[E.~Matsudo]{Eri Matsudo} 
\address{The Institute of Natural Sciences, Nihon University, 3-25-40 Sakurajosui, Setagaya-ku, Tokyo 156-8550, Japan}
\email{matsudo.eri@nihon-u.ac.jp}

\author[K.~Oshiro]{Kanako Oshiro}
\address{Department of Information and Communication Sciences, Sophia University, Tokyo 102-8554, Japan}
\email{oshirok@sophia.ac.jp}

\author[G.~Yamagishi]{Gaishi Yamagishi}
\address{}
\email{g-yamagishi-3c9@eagle.sophia.ac.jp}



\date{\today}

\maketitle

\begin{abstract}
This is the second paper which discusses minimum numbers of ``region" colors for knots, while minimum numbers of arc colors are well-studied, where the first one is our previous paper \cite{MatsudoOshiroYamagishi-1}.

In this paper, we give a method to evaluate minimum numbers of Dehn colors for knots by using symmetric local biquandle cocycle invariants.
We give answers to some questions arising as a consequence of our previous paper \cite{MatsudoOshiroYamagishi-1}. 
In particular, we show that there exist knots which are distinguished by minimum numbers of Dehn colors. 
\end{abstract}

\section{Introduction}
This is the second paper which discusses minimum numbers of ``region" colors for knots, while minimum numbers of arc colors are well-studied, where the first one is our previous paper \cite{MatsudoOshiroYamagishi-1}.

In knot theory, minimum numbers of colors of knots for Fox colorings are studied in many papers (see \cite{AbchirElhamdadiLamsifer, BentoLopes, HanZhou, HararyKauffman, IchiharaMatsudo, NakamuraNakanishiSatoh13,   NakamuraNakanishiSatoh16, Oshiro10, OshiroSatoh, Satoh09, Saito10} for example).
In particular, it is known that when $p$ is an odd prime number, for any Fox $p$-colorable knot $K$, the minimum number of colors of $K$, denoted by $\mincolF_p (K)$, satisfies that $\mincolF_p (K) = \lfloor \log_2 p \rfloor +2$ (see \cite{IchiharaMatsudo, NakamuraNakanishiSatoh13}).

Dehn colorings are a kind of region coloring for knot diagrams, which is known to be corresponding to Fox colorings.    
However, it is no exaggeration to say that minimum numbers of colors of knots for Dehn colorings (moreover those for region colorings) almost have not been studied yet. 
In our previous paper \cite{MatsudoOshiroYamagishi-1}, 
the next properties were shown.
\begin{itemize}
\item For any odd prime number $p$ and any Dehn $p$-colorable knot $K$, 
$$\mincol_p(K) \geq \lfloor \log_2 p \rfloor +2$$  (see also Theorem~\ref{th:MOY1-1}).
\item When $p=3$ or $5$, for any Dehn $p$-colorable knot $K$, $$\mincol_p(K) =\lfloor  \log_2 p \rfloor +2$$  (see also Remark~~\ref{rem:pcoloreddiagram}).
\item For any odd prime number $p$ with $p<2^5$ and $p\not \in \{13, 29\}$, 
there exists a Dehn $p$-colorable knot $K$ with $$\mincol_p(K) =\lfloor \log_2 p \rfloor +2$$ (see also Proposition~\ref{thm:pcoloreddiagram}).
\end{itemize}
Therefore, it is natural to ask the following questions.
\begin{question}\label{question1}
\begin{itemize}
\item[(1)] For each $p\in \{13, 29\}$, does there exist a Dehn $p$-colorable knot $K$ with $$\mincol_p(K) =\lfloor \log_2 p \rfloor +2~?$$
\item[(2)] For each prime number $p$ with $p>5$, does there exist a Dehn $p$-colorable knot $K$ with 
 $$\mincol_p(K) > \lfloor  \log_2 p \rfloor +2~?$$ 
\item[(3)] For each prime number $p$ with $p>5$, do there exist two Dehn $p$-colorable knots $K_1$ and $K_2$ with $$\mincol_p(K_1)\not = \mincol_p(K_2)?$$
\end{itemize}

\end{question}

In this paper, we will discuss some methods to evaluate $\mincol_p(K)$ by using symmetric local biquandle cocycle invariants (see Theorems~\ref{thm:main3} and \ref{thm:main2}).
In particular,  in Theorem~\ref{thm:main3}, we will prove that $\mincol_p(K) \geq \lfloor \log_2 p \rfloor +3$ when $p\in \{13, 29\}$, which implies that the answer to (1) of Question~\ref{question1} is ``NO''. 
Moreover, in the proof of Proposition~\ref{thm:main4}, we will give a concrete example of a Dehn $p$-colorable knot $K$ with $\mincol_p(K) \geq \lfloor  \log_2 p \rfloor +3$ for each prime number $p$ with $5<p<2^5$ and $p\not=17$, which implies that the answer to (2) of Question~\ref{question1} is ``Yes'' for some $p$. 
The results of this paper together with the results in \cite{MatsudoOshiroYamagishi-1} also give a partial answer to (3) of Question~\ref{question1}, i.e., the answer to (3) is "Yes" for some $p$. 
Note that for the case of $p=17$, the evaluation method introduced in this paper is not applied, and thus, we need anther method to study this case in our future. 

This paper is organized as follows.
In Section~2, we review the definitions and previous results on Dehn colorings and minimum numbers of colors of Dehn $p$-colorable knots. Symmetric local biquandle (co)homology groups and symmetric local biquandle cocyle invariants are defined in Sections~3 and 4, respectively. Section~5 is devoted to state our main results.
In Sections~6, 7 and  8, we prove the main results: Theorems~\ref{thm:main3}, \ref{thm:main2} and Proposition~\ref{thm:main4}, respectively.

\section{Dehn colorings and minimum numbers of colors}
In this paper, for an odd prime number $p$, we denote by $\mathbb Z_p$ the cyclic group $\mathbb Z/ p\mathbb Z$.

Let $p$ be an odd prime number.
Let $D$ be a diagram of a knot $K$, and $\mathcal{R}(D)$ the set of regions of $D$.
A {\it Dehn $p$-coloring} of $D$ is 
a map $C: \mathcal{R}(D) \to \mathbb Z_p$ 
satisfying the following condition: 
\begin{itemize}
\item for each crossing $\chi$ with regions 
$x_1, x_2, x_3$, and $x_4$ 
as depicted in Figure~\ref{coloring2},
\[
C(x_1) + C(x_3) =C(x_2) + C(x_4)
\]
holds, where the region $x_2$ is adjacent to $x_1$ by an under-arc and $x_3$ is adjacent to $x_1$ by the over-arc.
\end{itemize}
We call $C(x)$ the {\it color} of a region $x$ by $C$. 
In this paper, as shown in the right of Figure~\ref{coloring2}, we represent a Dehn $p$-coloring $C$ of a knot diagram $D$ by assigning the color $C(x)$ to each region $x$.  
We mean by $(D,C)$ a diagram $D$ given a Dehn $p$-coloring $C$, and call it a {\it Dehn $p$-colored diagram}. 
We denote by $\mathcal{C}(D, C)$ the set of colors assigned to a region of $D$ by $C$, that is $\mathcal{C}(D, C)={\rm Im}\,C$.
The set of Dehn $p$-colorings of $D$ is denoted by ${\rm Col}_{p}(D)$.
We remark that the number $\# {\rm Col}_{p}(D)$ is an invariant of the knot $K$.
\begin{figure}[t]
  \begin{center}
    \includegraphics[clip,width=7cm]{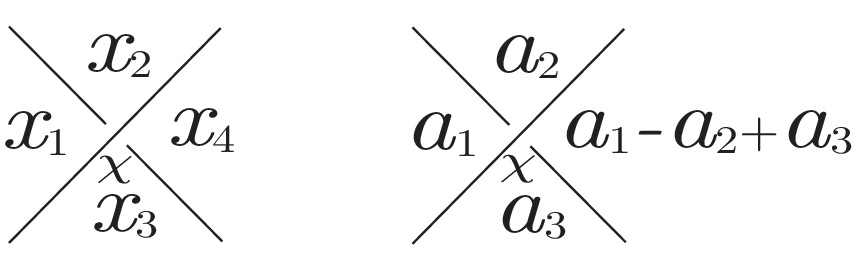}
    \caption{A crossing on $D$ and the one on $(D,C)$ with $C(x_1)=a_1$, $C(x_2)=a_2$, $C(x_3)=a_3$, $C(x_4)=a_1-a_2+a_3$}
    \label{coloring2}
  \end{center}
\end{figure}

For a semiarc $u$ of a Dehn $p$-colored diagram $(D,C)$, when the two regions $x_1$ and $x_2$ which are on the both side of $u$ have $C(x_1)=a_1$ and $C(x_2)=a_2$ for some $a_1,a_2\in \mathbb Z_p$, we call $u$ an {\it $\{a_1,a_2\}$-semiarc}, where $\{a_1,a_2\}$ is regarded as a multiset $\{a_1,a_1\}$ when $a_1=a_2$.

Let $\chi$ be a crossing of $D$ with regions $x_1, x_2, x_3$, and $x_4$ 
as depicted in Figure~\ref{coloring2}.
We say that $\chi$ of $(D,C)$ is {\it trivially colored}  if 
\[
C(x_1) = C(x_4), \mbox{ and } C(x_3) =C(x_2)
\]
hold, and {\it nontrivially colored} otherwise.
A Dehn $p$-coloring $C$ of $D$ is {\it trivial} if each crossing of $(D,C)$ is trivially colored, and {\it nontrivial} otherwise. 
In other words, a Dehn $p$-coloring $C$ of $D$ is {\it trivial} if it is a monochromatic coloring or a checkerboard coloring.

We denote by ${\rm Col}_p^{\rm NT}(D)$ the set of nontrivial Dehn $p$-colorings of $D$. 
We remark that the number $\# {\rm Col}_{p}^{\rm NT}(D)$ is an invariant of the knot $K$.
A knot $K$ is {\it Dehn $p$-colorable} if $K$ has a Dehn $p$-colored diagram $(D,C)$ such that $C$ is nontrivial.

The set ${\rm Col}_{p}(D)$ can be regarded as a $\mathbb Z_p$-module with the scalar product $sC$ and the addition $C+C'$ with
\begin{align*}
&(sC): \R(D) \to \mathbb Z_p; (sC)(x) = sC(x),\\ 
&(C+C'): \R(D) \to \mathbb Z_p; (C+C') (x) = C(x) + C'(x). 
\end{align*}
Moreover, since for any $t \in \mathbb Z_p$, the constant map $C^{\rm 1T}_t : \R(D) \to \mathbb Z_p; C^{\rm 1T} (x)=t $ is a kind of Dehn $p$-coloring, 
the map 
\[
(sC+t): \R(D) \to \mathbb Z_p; (sC+t)(x) = sC(x)+t,
\]
is also a Dehn $p$-coloring, and thus, the transformations $C \mapsto sC+t$ are closed in ${\rm Col}_p(D)$. 
In particular, a transformation $C \mapsto sC+t$ with $s \in \mathbb Z_p^\times$ is called a {\it regular affine transformation on ${\rm Col}_p(D)$}.
Two Dehn $p$-colorings $C, C' \in {\rm Col}_p(D)$ are {\it affine equivalent} (written $C\sim C'$) if they are related by a regular affine transformation on ${\rm Col}_p(D)$, that is, $C'=sC +t$ for some $s\in \mathbb Z_p^\times$ and $t \in \mathbb Z_p$.

A {\it regular affine transformation on the $\mathbb Z_p$-module $\mathbb Z_p$} is defined by $a\mapsto sa+t$ for some $s\in \mathbb Z_p^\times$ and $t \in \mathbb Z_p$.
Two subsets $S,S'\subset \mathbb Z_p$ are {\it affine equivalent}
(written $S\sim S'$) if they are related by a regular affine transformation on $\mathbb Z_p$, that is, $S'=sS +t$ for some $s\in \mathbb Z_p^\times$ and $t \in \mathbb Z_p$.
We note that since any regular affine transformation is a bijection, $\# S = \# S'$ holds if $S\sim S'$. 
\begin{lemma}[\cite{MatsudoOshiroYamagishi-1}]\label{lemma:colorequiv}
\begin{itemize}
\item[(1)] Let $C, C' \in  {\rm Col}_p(D)$. 
Then we have 
\[
C \sim C' \Longrightarrow \C(D,C) \sim \C(D,C').
\]
Hence we have 
\[
C \sim C'  \Longrightarrow \#\C(D,C) = \#\C(D,C').
\]
\item[(2)] Let $S, S' \subset \mathbb Z_p$, and we assume that $S\sim S'$. Then there exists $C\in {\rm Col}_p(D)$ such that $\C(D,C) = S$ if and only if  
there exists $C'\in {\rm Col}_p(D)$ such that $\C(D,C') = S'$
\end{itemize}
\end{lemma}
\begin{remark}\label{rem:colorequiv}
(1) of Lemma~\ref{lemma:colorequiv} implies that when we focus on the number of colors used for each nontrivial Dehn $p$-coloring of a diagram $D$, we may consider only the number of colors used for a representative of each affine equivalence class for nontrivial Dehn $p$-colorings of $D$.

(2) of Lemma~\ref{lemma:colorequiv} implies that when we focus on 
the sets of colors each of which might be used for a nontrivial Dehn $p$-coloring of a diagram, we may consider only representatives of the affine equivalence classes on the subsets of $\mathbb Z_p$.
\end{remark}

In this paper, we focus on the minimum number of colors.
\begin{definition} 
The {\it minimum number of colors} of a knot $K$ for Dehn $p$-colorings is the minimum number of distinct elements of $\mathbb Z_p$ which produce a nontrivially Dehn $p$-colored diagram of $K$, that is, 
\[
\min\Big\{\#\mathcal{C}(D,C) ~\Big|~ (C,D)\in \left\{
\begin{minipage}{4.3cm}
nontrivially Dehn $p$-colored\\
 diagrams of $K$
\end{minipage}
\right\}\Big\}.
\]
We denote it by $\mincol_p(K)$.
\end{definition}

In our previous paper \cite{MatsudoOshiroYamagishi-1}, the following properties were proven.
\begin{theorem}[\cite{MatsudoOshiroYamagishi-1}]\label{th:MOY1-1}
Let $p$ be an odd prime number. 
For any Dehn $p$-colorable knot $K$, we have  
\[
\mincol_p(K) \geq \lfloor \log_2 p \rfloor +2.
\]
\end{theorem}

\begin{theorem}[\cite{MatsudoOshiroYamagishi-1}]\label{thm:colorcandidate}
Let $p$ be an odd prime number with $p< 2^5$.
If there exists a nontrivially Dehn $p$-colored diagram $(D,C)$ of a knot such that $\#\C(D,C)= \lfloor \log_2 p \rfloor +2$, then 
\begin{itemize}
\item[(i)] $\C(D,C) \sim \{0,1,2\}$ when $p=3$,  
\item[(ii)] $\C(D,C) \sim \{0,1,2,3\}$ when $p=5$, 
\item[(iii)] $\C(D,C) \sim \{0,1,2,4\}$ when $p=7$, 
\item[(iv)] $\C(D,C) \sim \{0,1,2,3,6\}$ or $ \{0,1,2,4,7\}$ when $p=11$, 
\item[(v)] $\C(D,C) \sim \{0,1,2,4,7\}$ when $p=13$, 
\item[(vi)] $\C(D,C) \sim \{0,1,2,3,5,9\}$, $\{0,1,2,3,5,10\}$, $\{0,1,2,3,5,12\}$, $\{0,1,2,3,6,9\}$, $\{0,1,2,3,6,10\}$, $\{0,1,2,3,6,11\}$, $\{0,1,2,3,6,13\}$, $\{0,1,2,3,7,11\}$, $\{0,1,2,4,5,9\}$, $\{0,1,2,4,5,10\}$,  $\{0,1,2,4,5,12\}$,  or $\{0,1,2,4,10,13\}$ when $p=17$, 
\item[(vii)] $\C(D,C) \sim \{0,1,2,3,5,10\}$, $\{0,1,2,3,6,10\}$, $\{0,1,2,3,6,11\}$, $\{0,1,2,3,6,12\}$, $\{0,1,2,3,6,13\}$, $\{0,1,2,3,6,14\}$, $\{0,1,2,3,7,12\}$, $\{0,1,2,4,5,10\}$, $\{0,1,2,4,5,14\}$, $\{0,1,2,4,7,12\}$, or $\{0,1,2,4,7,15\}$ when $p=19$, 
\item[(viii)] $\C(D,C) \sim \{0,1,2,3,6,12\}$, $\{0,1,2,4,7,12\}$, $\{0,1,2,4,7,13\}$, $\{0,1,2,4,7,14\}$, $\{0,1,2,4,9,14\}$, or $\{0,1,2,4,10,19\}$ when $p=23$, 
\item[(ix)] $\C(D,C) \sim \{0,1,2,4,8,15\}$ when $p=29$, and 
\item[(x)] $\C(D,C) \sim \{0,1,2,4,8,16\}$ when $p=31$.
\end{itemize}
\end{theorem}
\begin{proposition}[\cite{MatsudoOshiroYamagishi-1}]\label{thm:pcoloreddiagram}
For each odd prime number $p$ with $p<2^5$ and $p\not \in \{13, 29\}$, there exists a Dehn $p$-colorable knot $K$ with $\mincol_p(K) =\lfloor \log_2 p \rfloor +2$. 
\end{proposition}

\begin{remark}[\cite{MatsudoOshiroYamagishi-1}]\label{rem:pcoloreddiagram}
The following properties hold.
\begin{itemize}
\item[(1)] When $p=3$, any $p$-colorable knot $K$ has $\mincol_p(K) =3$. 
\item[(2)] When $p=5$, any $p$-colorable knot $K$ has $\mincol_p(K) =4$. 
\end{itemize}
\end{remark}

\section{Local biquandle (co)homology groups useful for unoriented knots}\label{Sec:homology}

We first remark that the contents discussed in this section, except for Lemmas~\ref{lem:chaingroup} and \ref{lem:cocycle}, can be also applied in more general for a knot-theoretic horizontal-ternary-quasigroup $(X,[\,])$, see \cite{NelsonOshiroOyamaguchi}, with a tribracket $[\,]$ satisfying that  
\begin{itemize}
\item for any $a,b,c\in X$, 
\begin{align}\label{eq1}
[b,a,[a,b,c]]=c \mbox{ \ \ and \ \  } [c,[a,b,c], a] =b,
\end{align} 
\end{itemize} 
where this condition is required for dealing with ``unoriented'' knots.
See Remark~\ref{Rem:localbiquandle} for more details.

In this section, we introduce a local biquandle  (co)homology groups with an involution $\rho$ for the case related to Dehn colorings, where 
an ordinary version of local biquandle (co)homology groups was defined in \cite{NelsonOshiroOyamaguchi}. 
We remark that while an ordinary version is useful to construct an invariant for oriented knots, this version is useful to construct an invariant for ``unoriented'' knots (see Section~\ref{sec:cocycleinvariants}).

Let $p$ be an odd prime number, and put $X=\mathbb Z_p$ and $[a,b,c]:=a-b+c$. We define $\rho: X^2 \to X^2$ by $\rho \big( (a,b) \big) = (b,a)$.
For each $a \in X$, we define two operations $\ustar_a, \ostar_a: (\{a\} \times X)^2 \to X^2$ by 
\[
\begin{array}{l}
(a,b) \ustar_a (a,c) = (c, [a,b,c]), \mbox{ and }\\[5pt]
(a,b) \ostar_a (a,c) = (c, [a,c,b]). 
\end{array}
\]
We call $(X, \{\ustar_a\}_{a\in X}, \{\ostar_a\}_{a\in X})$ the {\it local biquandle associated with $(X,[\, ])$}.
In this paper, for simplicity, we often omit the subscript by $a$ as $\ustar=\ustar_a$, $\ostar=\ostar_a$,  
$\{\ustar\}=\{\ustar_a\}_{a\in X} $,  and 
$\{\ostar\}=\{\ostar_a\}_{a\in X} $ unless it causes confusion.

Let $n \in \mathbb Z$.
Let $C^{\rm lb}_n(X)$ be the free $\mathbb Z$-module generated by the elements of 
\[
\bigcup_{a\in X} (\{a\} \times X)^n =\big\{\big( (a,b_1), (a,b_2), \ldots , (a,b_n) \big) ~|~ a, b_1, \ldots , b_n \in X \big\}
\]
if $n\geq 1$, and $C^{\rm lb}_n(X)=0$ otherwise.
We define a homomorphism $\partial_n^{\rm lb} : C_n^{\rm lb} (X) \to C_{n-1}^{\rm lb} (X)$ by 
\begin{align}
&\partial_n^{\rm lb} \Big( \big( (a,b_1), \ldots , (a,b_n) \big) \Big) = \sum_{i=1}^{n} (-1)^i \big\{ \big( (a,b_1), \ldots, (a,b_{i-1}), (a,b_{i+1}), \ldots  , (a,b_n) \big) \notag \\
&- \big( (a,b_1)\ustar (a,b_i), \ldots, (a,b_{i-1})\ustar (a,b_i),    (a,b_{i+1})\ostar (a,b_i) ,\ldots  , (a,b_n) \ostar (a,b_i)\big) \big\} \notag \\
 &= \sum_{i=1}^{n} (-1)^i \big\{ \big( (a,b_1), \ldots,  (a,b_{i-1}), (a,b_{i+1}),  \ldots  , (a,b_n) \big) \notag \\
&- \big( (b_i, [a,b_1, b_i] ), \ldots, (b_i, [a,b_{i-1}, b_i]),    (b_i, [a,b_i, b_{i+1}]) ,\ldots  , (b_i, [a,b_i, b_{n}])\big) \big\} \notag 
\end{align}
if $n\geq 2$, and $\partial_n^{\rm lb}=0$ otherwise. Then $C_*^{\rm lb}(X)=\{C_n^{\rm lb}(X), \partial_n^{\rm lb}\}_{n\in \mathbb Z}$ is a chain complex (see \cite{NelsonOshiroOyamaguchi}).
Let $D_n^{\rm lb}(X)$ be the submodule of $C_n^{\rm lb}(X)$ that is generated by the elements of 
\[
\Big\{\big( (a,b_1),  \ldots , (a,b_n) \big) \in \bigcup_{a\in X
}  (\{a\} \times X)^n ~\Big|~ \mbox{ $b_i =b_{i+1}$ for some $i\in \{1, \ldots , n-1 \} $  }  \Big\}.
\]
Then 
$D_*^{\rm lb}(X)=\{D_n^{\rm lb}(X), \partial_n^{\rm lb}\}_{n\in \mathbb Z}$ is a subchain complex of $C_*^{\rm lb}(X)$ (see \cite{NelsonOshiroOyamaguchi}).
Besides, we define $D_n^{\rm lb}(X,\rho)$ to be the submodule of $C_n^{\rm lb}(X)$ that is generated by the elements of 
\[
\left\{
\begin{array}{l}
\big( (a,b_1),  \ldots , (a,b_n) \big) +\\  
\big( (a,b_1)\ustar (a,b_i) ,  \ldots , (a,b_{i-1}) \ustar (a,b_i)  ,\\
  \rho\big( (a,b_i) \big), (a,b_{i+1}) \ostar (a,b_i) , \ldots , (a,b_n)\ostar (a,b_i)  \big)
\end{array} 
~\Bigg|~ 
\begin{array}{l}
\big( (a,b_1),  \ldots , (a,b_n) \big) \\
\in \bigcup_{a\in X}  (\{a\} \times X)^n, \\
i\in \{1, \ldots , n\} 
\end{array}
 \right\}.
\]
We then have the following lemma.
\begin{lemma}
$D_*^{\rm lb}(X,\rho)=\{D_n^{\rm lb}(X,\rho), \partial_n^{\rm lb}\}_{n\in \mathbb Z}$ is a subchain complex of $C_*^{\rm lb}(X)$.
\end{lemma}
\begin{proof} 
To shorten notation, we put $ab:=(a,b)$. 

First, we have 
\begin{align}
&\partial_n^{\rm lb} \Big( \big( ab_1, \ldots , ab_n \big)  \Big) \notag\\
&= \sum_{j=1}^{i-1} (-1)^j  \big(  \ldots, ab_{j-1}, ab_{j+1}, \ldots, ab_i, \ldots   \big) \label{eq:subchain1-1} \\
&\hspace{1em}+ (-1)^i  \big(  \ldots, ab_{i-1}, ab_{i+1}, \ldots   \big)\label{eq:subchain1-2} \\
&\hspace{1em}+\sum_{j=i+1}^{n} (-1)^j  \big(  \ldots, ab_i,  \ldots, ab_{j-1}, ab_{j+1},  \ldots \big) \label{eq:subchain1-3} \\
&\hspace{1em}-\sum_{j=1}^{i-1} (-1)^j \big(  \ldots, ab_{j-1}\ustar ab_j,    ab_{j+1}\ostar ab_j ,\ldots, ab_{i}\ostar ab_j ,\ldots  \big)  \label{eq:subchain1-4} \\
&\hspace{1em}- (-1)^i \big(  \ldots, ab_{i-1}\ustar ab_i,    ab_{i+1}\ostar ab_i ,\ldots  \big) \label{eq:subchain1-5} \\
&\hspace{1em}-\sum_{j=i+1}^{n} (-1)^j \big(  \ldots, ab_{i}\ustar ab_j ,\ldots, ab_{j-1}\ustar ab_j,    ab_{j+1}\ostar ab_j ,\ldots  \big)  \label{eq:subchain1-6} 
\end{align}
and 
\begin{align}
&\partial_n^{\rm lb} \Big( 
\big(  ab_{1} \ustar ab_i,   \ldots , ab_{i-1} \ustar ab_i  ,
  \rho( ab_i ), ab_{i+1} \ostar ab_i , \ldots , ab_{n} \ostar ab_i  \big)
 \Big) \notag\\
&=  \sum_{j=1}^{i-1} (-1)^j  \big(  \ldots, ab_{j-1}\ustar ab_i, ab_{j+1}\ustar ab_i, \ldots, ab_{i-1} \ustar ab_i,  \rho(ab_i), ab_{i+1} \ostar ab_i,  \ldots   \big) \label{eq:subchain2-1}\\
&\hspace{1em}+ (-1)^i  \big(  \ldots, ab_{i-1}\ustar ab_i, ab_{i+1}\ostar ab_i, \ldots   \big) \label{eq:subchain2-2}\\
&\hspace{1em}+\sum_{j=i+1}^{n} (-1)^j  \big( \ldots, ab_{i-1} \ustar ab_i,  \rho(ab_i), ab_{i+1} \ostar ab_i, \ldots, ab_{j-1}\ostar ab_i, ab_{j+1}\ostar ab_i , \ldots   \big)\label{eq:subchain2-3}\\
&\hspace{1em}-\sum_{j=1}^{i-1} (-1)^j \big(  \ldots, (ab_{j-1}\ustar ab_i)\ustar(ab_{j}\ustar ab_i),    (ab_{j+1}\ustar ab_i)\ostar(ab_{j}\ustar ab_i) ,\ldots, \notag \\
&\hspace{1em}(ab_{i-1}\ustar ab_i)\ostar(ab_{j}\ustar ab_i),  \rho(ab_{i})\ostar(ab_{j}\ustar ab_i), (ab_{i+1}\ostar ab_i)\ostar(ab_{j}\ustar ab_i) ,\ldots  \big)  \label{eq:subchain2-4}\\
&\hspace{1em}-(-1)^i \big(  \ldots, (ab_{i-1}\ustar ab_i)\ustar\rho(ab_{i}),    (ab_{i+1}\ostar ab_i)\ostar\rho(ab_{i}) ,\ldots  \big) \label{eq:subchain2-5}\\
&\hspace{1em}-\sum_{j=i+1}^{n} (-1)^j \big(  \ldots, (ab_{i-1}\ustar ab_i)\ustar(ab_{j}\ostar ab_i), \rho(ab_{i})\ustar(ab_{j}\ostar ab_i), \notag\\ 
&\hspace{1em}(ab_{i+1}\ostar ab_i)\ustar(ab_{j}\ostar ab_i) , \ldots, (ab_{j-1}\ostar ab_i)\ustar(ab_{j}\ostar ab_i),  \notag\\  
&\hspace{1em}
(ab_{j+1}\ostar ab_i)\ostar(ab_{j}\ostar ab_i) ,\ldots  \big) . \label{eq:subchain2-6} 
\end{align}
Then by direct calculation, we can see that
\begin{align}
&(\ref{eq:subchain1-1})+(\ref{eq:subchain2-1}), (\ref{eq:subchain1-3})+(\ref{eq:subchain2-3}), (\ref{eq:subchain1-4})+(\ref{eq:subchain2-4}), (\ref{eq:subchain1-6})+(\ref{eq:subchain2-6}) \in D_{n-1}^{\rm lb}(X,\rho),  \notag\\ &(\ref{eq:subchain1-2})+(\ref{eq:subchain2-5}) = 0, \mbox{ and }  (\ref{eq:subchain1-5})+(\ref{eq:subchain2-2}) = 0. \notag
\end{align}
Hence, 
\begin{align}
&\partial_n^{\rm lb} \Big( \big( ab_1, \ldots , ab_n \big)  \Big) \notag\\
&\hspace{1em}+\partial_n^{\rm lb} \Big( 
\big(  ab_{1} \ustar ab_i,   \ldots , ab_{i-1} \ustar ab_i  ,
  \rho( ab_i ), ab_{i+1} \ostar ab_i , \ldots , ab_{n} \ostar ab_i  \big)
 \Big) \notag\\
 &=((\ref{eq:subchain1-1})+(\ref{eq:subchain2-1}))+((\ref{eq:subchain1-3})+(\ref{eq:subchain2-3}))+((\ref{eq:subchain1-4})+(\ref{eq:subchain2-4}))+((\ref{eq:subchain1-6})+(\ref{eq:subchain2-6})) \notag\\
 & \in D_{n-1}^{\rm lb}(X,\rho).  \notag
\end{align}
This implies that $\partial_n^{\rm lb} \big(D_n^{\rm lb}(X,\rho)\big) \subset D_{n-1}^{\rm lb}(X,\rho)$, and thus, $D_*^{\rm lb}(X,\rho)$ is a subchain complex of $C_*^{\rm lb}(X)$.
\end{proof}

Therefore the chain complexes
$$C_*^{\rm LB} (X)=\{C_n^{\rm LB}(X):=C_n^{\rm lb}(X)/D_n^{\rm lb}(X), \partial_n^{\rm LB}:=\partial_n^{\rm lb}\}_{n\in \mathbb Z}$$
and 
$$C_*^{\rm SLB} (X)=\{C_n^{\rm SLB}(X):=C_n^{\rm lb}(X)/(D_n^{\rm lb}(X)+D_n^{\rm lb}(X,\rho)), \partial_n^{\rm SLB}:=\partial_n^{\rm lb}\}_{n\in \mathbb Z}$$ 
are induced.
The homology groups 
$H_n^{\rm LB} (X)$ of $C_*^{\rm LB} (X)$ and 
$H_n^{\rm SLB} (X)$ of $C_*^{\rm SLB} (X)$ are called the \textit{$n$th local biquandle homology group} and  \textit{$n$th symmetric local biquandle homology group}, respectively, of $(X, \{\ustar\} , \{ \ostar\})$ or 
$(X,[\, ])$.

For an abelian group $A$, we define the cochain complexes 
$C_{\rm LB}^\ast(X, A):=\{C_{\rm LB}^n(X, A), \delta^n_{\rm LB}\}_{n\in \mathbb Z}$ and 
 $C_{\rm SLB}^\ast(X, A):=\{C_{\rm SLB}^n(X, A), \delta^n_{\rm SLB}\}_{n\in \mathbb Z}$ by 
\[
\begin{array}{l}
C_{\rm LB}^n(X, A) ={\rm Hom}(C_n^{\rm LB}(X), A), \\
\delta^n_{\rm LB} : C_{\rm LB}^n(X, A) \to C_{\rm LB}^{n+1}(X, A); \delta^n_{\rm LB}(f)=f \circ \partial_{n+1}^{\rm LB},\\
C_{\rm SLB}^n(X, A)={\rm Hom}(C_n^{\rm SLB}(X), A), \\
\delta^n_{\rm SLB} : C_{\rm SLB}^n(X, A) \to C_{\rm SLB}^{n+1}(X, A); \delta^n_{\rm SLB}(f)=f \circ \partial_{n+1}^{\rm SLB}.
\end{array}
\]
The \textit{nth cohomology groups} $H^n_{\rm LB}(X, A)$ and $H^n_{\rm SLB}(X, A)$ of $(X, \{\ustar\}, \{\ostar\})$ or $(X, [\,])$ with coefficient group $A$ are defined by $H_{\rm LB}^n(X, A)=H^n(C^\ast_{\rm LB}(X, A))$ and 
$H_{\rm SLB}^n(X, A)=H^n(C^\ast_{\rm SLB}(X, A))$, respectively.

Note again that the next lemmas hold in only the case that $(X,[\,])=(\mathbb Z_p, [\,])$ with $[a,b,c]=a-b+c$. (Those can not be  satisfied generally for knot-theoretic horizontal-ternary-quasigroups.)
Before we state the following lemma, we define the magnitude relation on $\mathbb Z_p$ and $\mathbb Z_p^4$.
We note that the elements of $\mathbb Z_p$ is represented by the integers $0,1, \ldots , p-1$. 
The relation $<$ on $\mathbb Z_p$ are defined by the ordinary magnitude relation $<$ on $\mathbb Z$ for the representatives $0,1, \ldots , p-1$ regarded as integers. 
For two different $4$-tuples $(a_1,a_2,a_3,a_4), (b_1,b_2,b_3,b_4) \in \mathbb Z_p^4$, we define the order $<$ by 
\begin{align*}
&(a_1,a_2,a_3,a_4) < (b_1,b_2,b_3,b_4) \\
&\iff   a_j<b_j \mbox{ for } j= \min\{ i\in \{1,\ldots , 4\} ~|~ a_i \not = b_i \}.
\end{align*}
\begin{lemma}\label{lem:chaingroup}
It holds that 
\begin{align}
C_2^{\rm SLB} (X) \cong
& \bigoplus_{\begin{minipage}{4.2cm}
{\tiny {\rm 
$a,b,c\in \mathbb Z_p$
 s.t. $a\not=b$, $b\not=c$, \\
$(a,b,c,[a,b,c]) = \min Q(a,b,c)$
}
}
\end{minipage}
}\mathbb Z \left\langle \big( (a,b), (a,c) \big) \right\rangle \nonumber\\
&  \oplus \bigoplus_{\begin{minipage}{1.7cm}
{\tiny
$a,b\in \mathbb Z_p$
 {\rm s.t.} 
$a<b$}
\end{minipage}} \mathbb Z_2 \left\langle \big( (a,a), (a,b) \big) \right\rangle ,  \label{eq:C2}\\
C_1^{\rm SLB} (X) \cong & \bigoplus_{\begin{minipage}{1.7cm}
{\tiny
$a,b\in \mathbb Z_p$
 {\rm s.t.} 
$a<b$}
\end{minipage}
} \mathbb Z  \left\langle  (a,b) \right\rangle  \oplus \bigoplus_{\begin{minipage}{1cm}
{\tiny
$a\in \mathbb Z_p$}
\end{minipage}} \mathbb Z_2 \left\langle  (a,a)  \right\rangle, \label{eq:C1}
\end{align}
where 
\begin{align*}
&Q(a,b,c)\\
&= \{ (a,b,c, [a,b,c]), (b,a,[a,b,c],c), (c,[a,b,c],a,b), ([a,b,c],c,b,a)\},
\end{align*}
 and where $\mathbb Z \langle x \rangle$ represents 
the abelian group $\{n x \,|\, n\in \mathbb Z\}$, that is, 
the infinite cyclic group $\mathbb Z$  which is generated by $x$, and $\mathbb Z_2 \langle x \rangle$ represents the abelian group $\{n x \,|\, n\in \mathbb Z_2\}$, that is, the cyclic group $\mathbb Z_2$ which is generated by $x$.
\end{lemma}
\begin{proof}
Let $A$ denote the right-hand side of the equality (\ref{eq:C2}). 
We define a homomorphism $\Phi: C_2^{\rm SLB}(X) \to A$ as follows.
For $\big((a,b), (a,c)\big)\in  C_2^{\rm SLB}(X) $ with $a \not = b$, set
\begin{align*}
&\Phi\Big(\big((a,b), (a,c)\big)\Big)\\
&=\left\{
\begin{array}{ll}
+\big((a,b), (a,c)\big) &\mbox{if } b\not = c, \min Q(a,b,c)= (a,b,c,[a,b,c]),\\
-\big((b,a),(b,[a,b,c])\big)&\mbox{if } b\not = c, \min Q(a,b,c)=(b,a,[a,b,c],c),\\
-\big((c,[a,b,c]),(c,a)\big)&\mbox{if } b\not = c, \min Q(a,b,c)=(c,[a,b,c],a,b),\\
+\big(([a,b,c],c),([a,b,c],b)\big)&\mbox{if } b\not = c, \min Q(a,b,c)=([a,b,c],c,b,a)\\
0& \mbox{otherwise}, 
\end{array}
\right.
\end{align*}
and for $\big((a,a), (a,b)\big)\in  C_2^{\rm SLB}(X) $,  set
\begin{align*}
&\Phi\Big(\big((a,a), (a,b)\big)\Big)=\left\{
\begin{array}{ll}
+\big((a,a), (a,b)\big) &\mbox{if } a< b,\\
+\big((b,b),(b,a)\big)&\mbox{if } a>b,\\
0& \mbox{otherwise}. 
\end{array}
\right.
\end{align*}
Then we can see that $\Phi$ is an isomorphism.  
We leave the detailed proof to the reader.

Similarly we can also see the equality (\ref{eq:C1}). 
\end{proof}

\begin{lemma}\label{lem:cocycle}
The homomorphism $\theta_p: C_2^{\rm SLB} (X) \to \mathbb Z_p$ defined by 
\begin{align*}
\theta_p \Big(\big((a,b), (a,c)\big)\Big) &=(a-b) \frac{  (a-b+2c)^p +( a+b)^p- 2( a+c)^p   }{p}
\end{align*}
is a $2$-cocycle of $(X,[\,])=(\mathbb Z_p,[\,])$ with $[a,b,c]=a-b+c$, where for the fraction, the numerator is calculated
in $\mathbb  Z$ and it is divisible by $p$.
\end{lemma}
\begin{proof}
As discussed in \cite{Oshiro20},  the $2$-cocycle conditions  
\begin{itemize}
\item $\delta^2_{\rm SLB} (\theta_p) = 0$, and 
\item $\theta_p\Big(\big((a,b), (a,b)\big)\Big)=0$ for any $a,b\in X$
\end{itemize}
 are satisfied. 
Hence, it suffices to show that the cocycle conditions
\begin{align}
&\theta_p\Big(\big((a,b), (a,c)\big) + \big(\rho((a,b)), (a,c)\ostar (a,b) \big)\Big)=0, \mbox{ and}\label{equal:cocycleconditon1}\\
&\theta_p\Big(\big((a,b), (a,c)\big) + \big((a,b)\ustar (a,c), \rho((a,c)) \big)\Big)=0\label{equal:cocycleconditon2}
\end{align}
for any $a,b,c \in X$. 

Since
\begin{align*}
&\theta_p\Big(\big(\rho((a,b)), (a,c)\ostar (a,b) \big)\Big)\\
&=\theta_p\Big(\big((b,a), (b,a-b+c) \big)\Big)\\
&=(b-a) \frac{  (b-a+2(a-b+c))^p +(b+a)^p- 2(b+(a-b+c))^p   }{p}\\
&=-(a-b) \frac{  (a-b+2c)^p +( a+b)^p- 2( a+c)^p   }{p}\\
&= - \theta_p\Big(\big((a,b), (a,c)\big)\Big),
\end{align*}
the first equality (\ref{equal:cocycleconditon1}) holds, and the second equality  (\ref{equal:cocycleconditon2}) is also proven by the same method.

\end{proof}

\section{Local biquandle cocycle invariants of unoriented knots}\label{sec:cocycleinvariants}

We first remark that the contents discussed in this section except for Proposition~\ref{thm:cocycleinvariantforaffineequivalence} can be also applied in more general for a knot-theoretic horizontal-ternary-quasigroup $(X,[\,])$ with a tribracket $[\,]$ satisfying the condition (\ref{eq1}), and an $X$-colored diagram $(D,C)$ of a knot.
Here we note that ${\rm Col}_X(D)$ in the equation (\ref{eq:2}) and ${\rm Col}_X^{\rm NT}(D)$ in the equation (\ref{eq:3}) mean the set of $X$-colorings of a diagram $D$ and the set of nontrivial $X$-colorings of $D$, respectively, and moreover, they mean ${\rm Col}_p(D)$ and ${\rm Col}_p^{\rm NT}(D)$, respectively, for Dehn $p$-colorings. 
See Remark~\ref{Rem:localbiquandle} for more details.

Let $p$ be an odd prime number, and put $X=\mathbb Z_p$ and $[a,b,c]:=a-b+c$. We define $\rho: X^2 \to X^2$ by $\rho \big( (a,b) \big) = (b,a)$.
Let $(D, C)$ be a Dehn $p$-colored diagram of an (unoriented) knot $K$.

For each crossing $\chi$ of $(D,C)$, we take the weight $w_\chi$ as follows. There are locally the four regions around $\chi$. We choose one of them, and call it a {\it specified region}. 
Let $x_1$, $x_2$, $x_3$ and $x_4$ be the regions around $\chi$ such that $x_1$ is locally the specified region,  $x_2$ is the region adjacent to $x_1$ by an under-arc and $x_3$ is the region adjacent to $x_1$ by the over-arc as depicted in Figure~\ref{coloring5}, where we put $\star$ on the specified region for easy visual recognition. Put $a:=C(x_1)$, $b:=C(x_2)$ and $c:=C(x_3)$.
Then the weight of $\chi$ is defined by 
$w_\chi= \varepsilon \big((a,b), (a,c)\big)$, where for the normal vector $n_u$ from $x_1$ to $x_2$ and the normal vector $n_o$ from $x_1$ to $x_3$, $\varepsilon=+$ if $(n_o, n_u)$ matches the right-handed orientation of $\mathbb R^2$, and $\varepsilon=-$ otherwise (see Figure~\ref{coloring5}).
\begin{figure}[t]
  \begin{center}
    \includegraphics[clip,width=7cm]{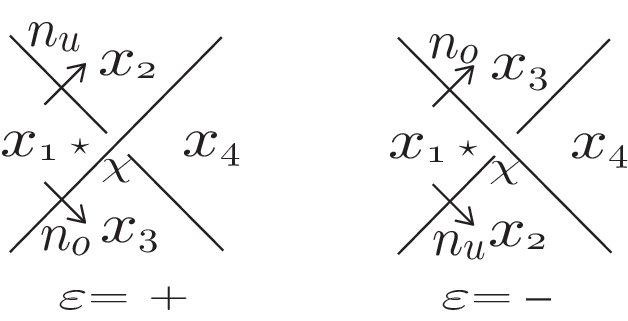}
    \caption{}
    \label{coloring5}
  \end{center}
\end{figure}

We denote by $W(D,C)$ the sum of the weights of all the crossings of $(D,C)$, and regard it an element of $C_2^{\rm SLB} (X)$. We then have the following lemmas.
\begin{lemma}\label{lem:weightsum1}
The value $W(D,C)$ in $C_2^{\rm SLB} (X)$ does not depend on the choice of specified regions for crossings.
\end{lemma} 
\begin{proof}
Let $\chi$ be a crossing of $(D,C)$ with regions colored by $a$, $b$, $c$ and $[a,b,c]$ as depicted in Figure~\ref{coloring3}. 
Then the weight of $\chi$ can be written in four ways as 
$w = +\big((a, b), (a, c)\big)$, $-\big((b, a), (b, [a,b,c])\big)$, $-\big((c, [a,b,c]), (c, a)\big)$, or  
$+\big(([a,b,c], c), ([a,b,c], b)\big)$,
depending on how the specified regions are choosen. 
 Since it holds that 
\begin{align}
&\big((a,b), (a,c)\big) + \big(\rho((a,b)), (a,c)\ostar(a,b)\big)=0, \notag\\ 
&\big((a,b), (a,c)\big) + \big((a,b)\ustar(a,c), \rho((a,c))\big)=0, \notag\\ 
&\big(([a,b,c],c), ([a,b,c],b)\big) + \big(\rho(([a,b,c],c)), ([a,b,c],b)\ostar([a,b,c],c)\big)=0 \notag
\end{align}
in $C_2^{\rm SLB} (X)$, we have
\begin{align}
&+\big((a, b), (a, c)\big) = -\big((b, a), (b, [a,b,c])\big) \notag\\
&= -\big((c, [a,b,c]), (c, a)\big) = +\big(([a,b,c], c), ([a,b,c], b)\big) \notag
\end{align}
in $C_2^{\rm SLB} (X)$.
 Therefore, the weight of a crossing does not depend on the choice of a specified regions. For the value $W(D,C)$, the same thing holds. 
 
 \begin{figure}[ht]
  \begin{center}
    \includegraphics[clip,width=10cm]{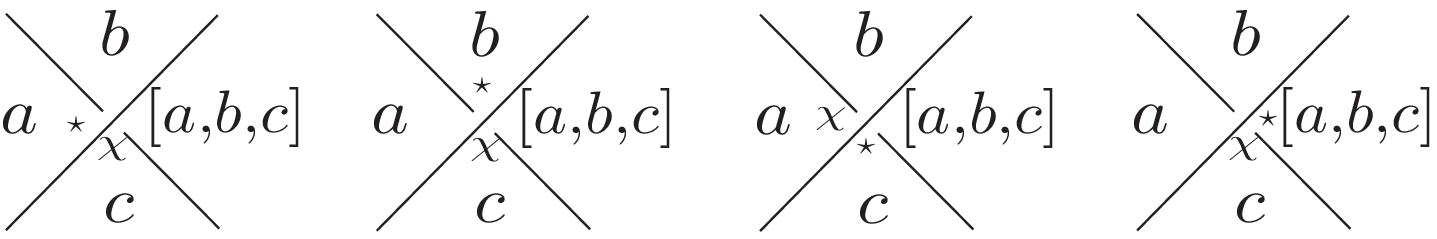}
  \end{center}
          \caption{}
            \label{coloring3}
\end{figure}
\end{proof}
\begin{lemma}\label{lem:weightsum2}
We have $W(D,C) \in \ker \partial_2^{\rm SLB}$.
\end{lemma} 
\begin{proof}
For a crossing $\chi$ as depicted in the upper of Figure~\ref{assigned to semi-arcs} , we have 
\begin{align}
&\partial_2^{\rm SLB} ( w_{\chi}  ) \notag\\
&=\partial_2^{\rm SLB} \Big(+ \big( (a, b), (a, c) \big)  \Big) \notag\\
&= -(a, c) + (a, b) + (b, [a, b, c]) - (c, [a, b, c]). \notag
\end{align}
The terms of this result can be assigned to the semi-arcs, around $\chi$, between the regions colored by the two elements of the $2$-tuple of each term as the upper of Figure~\ref{assigned to semi-arcs}. 
At the same time, we assign a normal vector to each semi-arc around $\chi$ so that it points from the region colored by the first entry of the assigned $2$-tuple. Note that we may change the assigned signed $2$-tuple, say $\varepsilon(a,b)$, and the normal vector, say $n_{a\to b}$ to $-\varepsilon(b,a)$ and the reversed orientation $n_{b\to a}$, since $\varepsilon(a,b)=-\varepsilon(b,a)$ in $C_1^{\rm SLB} (X)$. For a semi-arc, if the normal vector of the two ends of the semi-arc is different, we transform either one to the reversed orientation, and at the same time, the assigned signed $2$-tuple, say $\varepsilon(a,b)$, to $-\varepsilon(a,b)$.  
When we perform this operation for all semi-arcs, the two ends of each semi-arc have the same element of $X^2$, and their signs are different (see Figure~\ref{assigned to semi-arcs}). 
Thus, two terms of the result of $\partial_2^{\rm SLB} ( W(D,C)  )$ which are assigned to the same semi-arc are canceled. This implies that $\partial_2^{\rm SLB} ( W(D,C)  ) = 0$. 
\begin{figure}[h]
\begin{center}
\includegraphics[width=10cm]{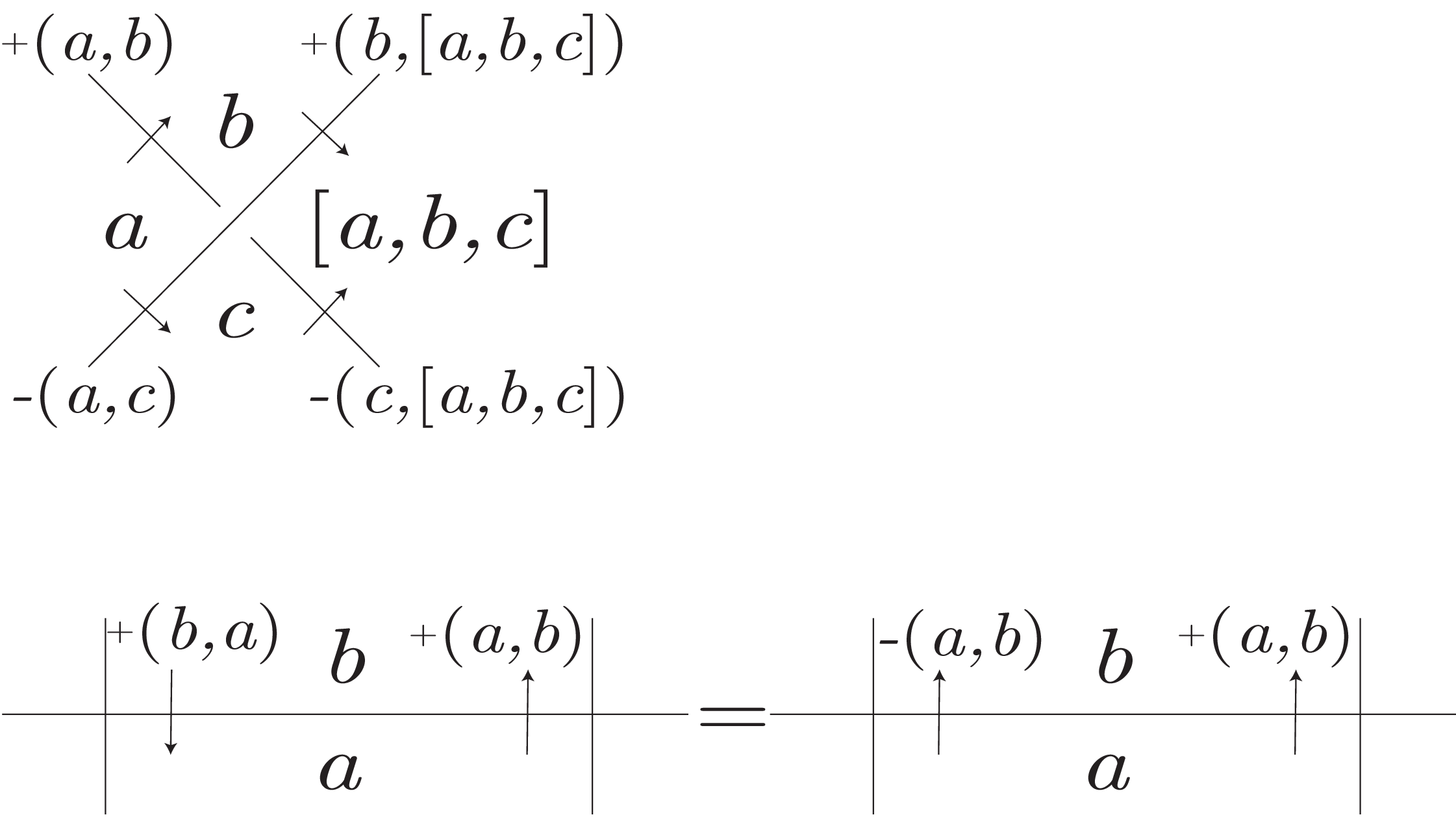}
\caption{} 
\label{assigned to semi-arcs}
\end{center}
\end{figure}
\end{proof}

For an Abelian group $A$, let $\theta: C_2^{\rm SLB}(X) \to A$ be a $2$-cocycle. 
We define 
\begin{align}\label{eq:2}
\Phi_{\theta}(D) :=\{ \theta(W(D,C)) ~|~ C\in {\rm Col}_X(D)\},
\end{align}
and 
\begin{align}\label{eq:3}
\Phi_{\theta}^{\rm NT}(D) :=\{ \theta(W(D,C)) ~|~ C\in {\rm Col}_X^{\rm NT}(D)\}
\end{align}
as multisets. 
We call them the {\it symmetric local biquandle cocycle invariants} of (unoriented) knots with respect to $(X, [\,])$ and $\theta$. 

\begin{theorem}
The multisets $\Phi_{\theta}(D)$ and $\Phi_{\theta}^{\rm NT}(D)$ are invariants of $K$. 
\end{theorem}
\begin{proof}
This can be proved similarly as in the case of the ordinary local biquandle cocycle invariants (see \cite{NelsonOshiroOyamaguchi}). 
\end{proof}

\begin{proposition}
Let $\theta, \theta': C_2^{\rm SLB}(X) \to A$ be $2$-cocycles.
If $\theta$ and $\theta'$ are cohomologous, $\Phi_{\theta}(K)=\Phi_{\theta'}(K)$ and $\Phi_{\theta}^{\rm NT}(K)=\Phi_{\theta'}^{\rm NT}(K)$ hold for any knot $K$.
\end{proposition}
\begin{proof}
Let $D$ be a diagram of $K$, and $C$ any nontrivial Dehn $p$-coloring of $D$. 
Since by the assumption, there exists a map $\varphi \in C^1_{\rm SLB}(X)$ such that $\theta -\theta' = \delta^1_{\rm SLB} (\varphi)=\varphi \circ \partial_2^{\rm SLB}$, we have
\begin{align*}
\theta(W(D,C)) - \theta'(W(D,C)) = \varphi \circ \partial_2^{\rm SLB} (W(D,C))=0,
\end{align*}
where the last equality follows from Lemma~\ref{lem:weightsum2}.

Thus we have $\theta(W(D,C)) = \theta'(W(D,C))$, which induces $\Phi_{\theta}^{\rm NT}(K)=\Phi_{\theta'}^{\rm NT}(K)$.
\end{proof}

\begin{proposition}\label{thm:cocycleinvariantforaffineequivalence}
Let $\theta_p$ be the $2$-cocycle defined in Lemma~\ref{lem:cocycle}.
Let $D$ be a diagram of a Dehn $p$-colorable knot. 
Let $C$ and $C'$ be Dehn $p$-colorings of $D$.
Suppose that $C'=sC+t$ for some $s\in \mathbb Z_p^{\times}$ and $t\in \mathbb Z_p$, that is, $C\sim C'$.  
Then we have 
\[
\theta_p(W(D,C'))=  s^2 \theta_p(W(D,C)).
\]
\end{proposition}
\begin{proof}

We remark that $\theta_p$ can be also regarded as an ordinary local biquandle $2$-cocycle $\theta_p: C_2^{\rm LB} (X) \to \mathbb Z_p$, because it satisfies the local biquandle $2$-cocycle condition (see Lemma~\ref{lem:cocycle} and \cite{Oshiro20}).
Moreover, since the value $\theta_p \Big(\big((a,b), (a,c)\big)\Big)$ for $\big((a,b), (a,c)\big)\in C_2^{\rm LB}(X)$ coincides with that for $\big((a,b), (a,c)\big)\in C_2^{\rm SLB}(X)$, we may discuss $\theta_p(W(D, sC+t))=s^2\theta_p(W(D, C))$ for $\theta_p \in C^2_{\rm LB}(X, \mathbb Z_p)$ and an oriented Dehn $p$-colored diagram $(D,C)$, which is the case that the specified region, for the weight of a crossing, is fixed depending on the orientation of $D$ as in the lower left of  Figure~\ref{coloring8} (See \cite{Oshiro20}). 

In \cite{Oshiro20} (see also \cite{NelsonOshiroOyamaguchi}), it is shown that the $2$-cocycle $\theta_p \in C^2_{\rm LB}(X, \mathbb Z_p)$ can be obtained from the Mochizuki's $3$-cocycle \cite{Mochizuki}, denoted by $\theta_p^{\rm M}$, of the dihedral quandle $R_p$ of order $p$ as
\begin{align*}
\theta_p\Big(\big((a,b), (a,c)\big)\Big) =2^2 \theta_p^{\rm M}\Big(\big(a, \frac{a+b}{2}, \frac{a+c}{2}\big)\Big),
\end{align*}
and, it is also shown that $\theta_p (W(D,C))=2^2\theta_p^{\rm M} (W(D,\oline{C}))$, where $\oline{C}$ is the shadow $R_p$-coloring corresponding to $C$ (see Figure~\ref{coloring8}).
Moreover, $\theta_p^{\rm M} (W(D,s\oline{C}+t)) = s^2 \theta_p^{\rm M} (W(D,\oline{C}))$ holds for affine equivalent shadow $R_p$-colorings $\oline{C}$ and $s\oline{C}+t$ of $D$ (see \cite{Satoh07}).
Therefore we have 
\begin{align*}
\theta_p (W(D,sC+t)) &= 2^2 \theta_p^{\rm M} (W(D,\oline{sC+t} )) \\
&= 2^2 \theta_p^{\rm M} (W(D,s\oline{C}+t ))\\
& =2^2  s^2 \theta_p^{\rm M} (W(D,\oline{C}))\\
& =  s^2 \theta_p(W(D,C)). 
\end{align*}
\begin{figure}[t]
  \begin{center}
    \includegraphics[clip,width=8cm]{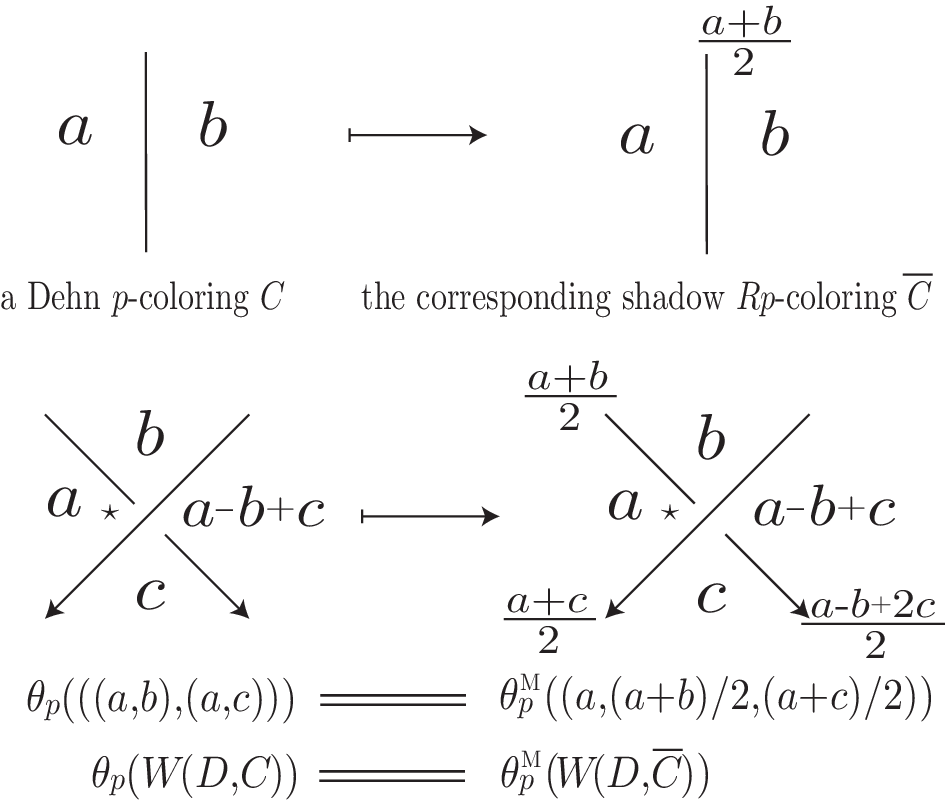}
    \caption{}
    \label{coloring8}
  \end{center}
\end{figure}
\end{proof}

\begin{remark}\label{Rem:localbiquandle}
In general, the (co)homology groups defined in Section~\ref{Sec:homology} and the cocycle invariants of (unoriented) knots defined in this section can be also defined for a knot-theoretic horizontal-ternary-quasigroup $(X,[\,])$, see \cite{NelsonOshiroOyamaguchi}, with a tribracket $[\,]$ satisfying the condition (\ref{eq1}).
Here $(X, [\,])$ gives a region coloring of an (unoriented) knot diagram $D$, that is, a map $C:\R(D) \to X$ satisfying the crossing condition
\[
[C(x_1), C(x_2), C(x_3)] = C(x_4)
\] 
as depicted in Figure~\ref{coloring7}, where we first choose a specified region around a crossing $\chi$, and denote the regions around $\chi$ by $x_1$, $x_2$, $x_3$, and $x_4$ as above.   
Thanks to the condition (\ref{eq1}), the crossing condition is independent of the choice of the specified region (see Figure~\ref{coloring6}). 
\end{remark}

 \begin{figure}[ht]
  \begin{center}
    \includegraphics[clip,width=6cm]{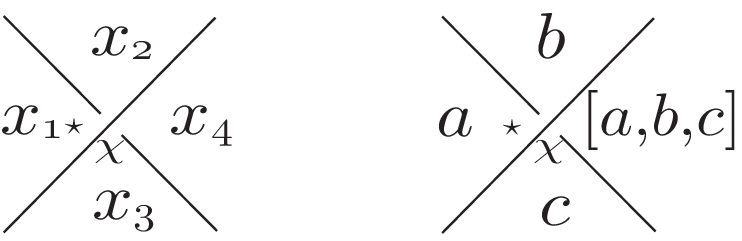}
  \end{center}
          \caption{}
            \label{coloring7}
\end{figure}
\begin{figure}[ht]
  \begin{center}
    \includegraphics[clip,width=8cm]{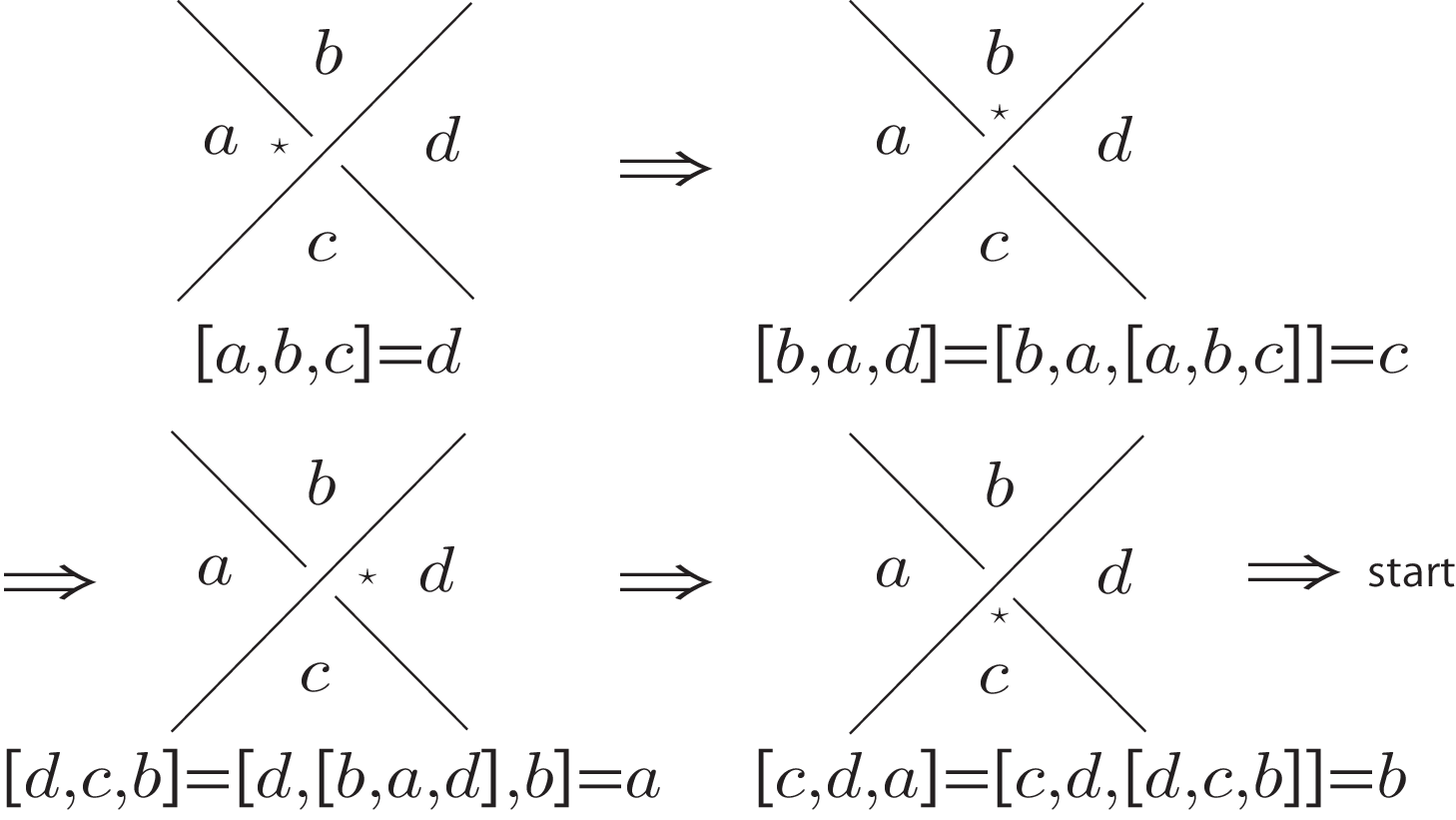}
  \end{center}
      \caption{}
    \label{coloring6}
\end{figure}

\section{New evaluation results for  $\mincol_p(K)$}
In this section, we will give new evaluation results with respect to $\mincol_p(K)$ for a knot $K$ and an odd prime number $p$ with $5<p< 2^5$.
Here, we recall again Question~\ref{question1}.
\begin{itemize}
\item[(1)] For each $p\in \{13, 29\}$, does there exist a Dehn $p$-colorable knot $K$ with $$\mincol_p(K) =\lfloor \log_2 p \rfloor +2~?$$
\item[(2)] For each odd prime number $p$ with $5<p<2^5$, does there exist a Dehn $p$-colorable knot $K$ with 
 $$\mincol_p(K) > \lfloor  \log_2 p \rfloor +2~?$$ 
\item[(3)] For each prime number $p$ with $p>5$, do there exist two Dehn $p$-colorable knots $K_1$ and $K_2$ with $$\mincol_p(K_1)\not = \mincol_p(K_2)?$$
\end{itemize}

The next theorem gives that the answer of (1) of Question~\ref{question1} is ``No''.
\begin{theorem}\label{thm:main3}
When $p\in \{13, 29\}$, for any Dehn $p$-colorable knot $K$,  
\[
\mincol_p(K) \geq \lfloor \log_2 p \rfloor +3
\]
holds.
\end{theorem}

For an odd prime number $p$, let $\theta_p: C_2^{\rm SLB} (\mathbb Z_p) \to \mathbb Z_p$  be the $2$-cocycle of $(\mathbb Z_p, [\,])$ with $[a,b,c]=a-b+c$ defined in Lemma~\ref{lem:cocycle}, that is, defined by 
\begin{align*}
\theta_p \Big(\big((a,b), (a,c)\big)\Big) &=(a-b) \frac{  (a-b+2c)^p +( a+b)^p- 2( a+c)^p   }{p}.
\end{align*}

We can evaluate $\mincol_p(K)$ by using  symmetric local biquandle cocycle invariants as follows.
\begin{theorem}\label{thm:main2}
Let $p$ be an odd prime number with $5<p<2^5$ and $p\not =17$.
Let $K$ be a Dehn $p$-colorable knot. 

If the symmetric local biquandle cocycle invariant $\Phi_{\theta_p}^{\rm NT}(K)$ satisfies $0 \not\in \Phi_{\theta_p}^{\rm NT}(K)$, then we have 
\[
\mincol_p(K) \geq \lfloor \log_2 p \rfloor +3.
\]

\end{theorem}

The next proposition partially gives that the answer to (2) of Question~\ref{question1} is ``Yes'' for some $p$.
This also gives a partial answer to (3) of Question~\ref{question1} together with the results in \cite{MatsudoOshiroYamagishi-1}, i.e., the answer to (3) is ``Yes" for some $p$.. 

\begin{proposition}\label{thm:main4} 
Let $p$ be an odd prime number with $5<p<2^5$ and $p\not =17$.
Then there exists a Dehn $p$-colorable knot $K$ with 
\[
\mincol_p(K) \geq \lfloor \log_2 p \rfloor +3.
\]
In particular, when $p\in \{7, 11, 13, 19, 29\}$, we have a Dehn $p$-colorable knot $K$ with 
\[
\mincol_p(K) = \lfloor \log_2 p \rfloor +3,
\]
and when $p\in \{23,31\}$, we have a Dehn $p$-colorable knot $K$ with
\[
\mincol_p(K) = \lfloor \log_2 p \rfloor +3 \mbox{ or } \lfloor \log_2 p \rfloor +4.
\]  
\end{proposition}

\begin{remark}
When $p=17$, Theorem~\ref{thm:main2} does not hold, because 
there exists a Dehn $17$-colorable knot $K$ with $0 \not\in \Phi_{\theta_{17}}^{\rm NT}(K)$ and  $\mincol_{17}(K)=\lfloor \log_2 17 \rfloor +2= 6$, which is shown below more precisely.

Let $K$ be the knot represented by the diagram $D$ in Figure~\ref{p=17diagex}.
As shown in Figure~\ref{p=17diagex}, there exists a Dehn $17$-coloring $C$ with $\#\C(D,C)=6$, which together with Theorem~\ref{th:MOY1-1} implies that $\mincol_{17}(K)=6$. 

On the other hand, since 
\begin{align*}
\Phi_{\theta_{17}}^{\rm NT}(K)=&\{1 (578\mbox{ times}), 2 (578\mbox{ times}), 4 (578\mbox{ times}), 8 (578\mbox{ times}), \\
&9 (578\mbox{ times}), 13 (578\mbox{ times}), 15 (578\mbox{ times}), 16 (578\mbox{ times})\}, 
\end{align*}
$0 \not\in \Phi_{\theta_{17}}^{\rm NT}(K)$ holds.
\begin{figure}[h]
\begin{center}
\includegraphics[width=6cm]{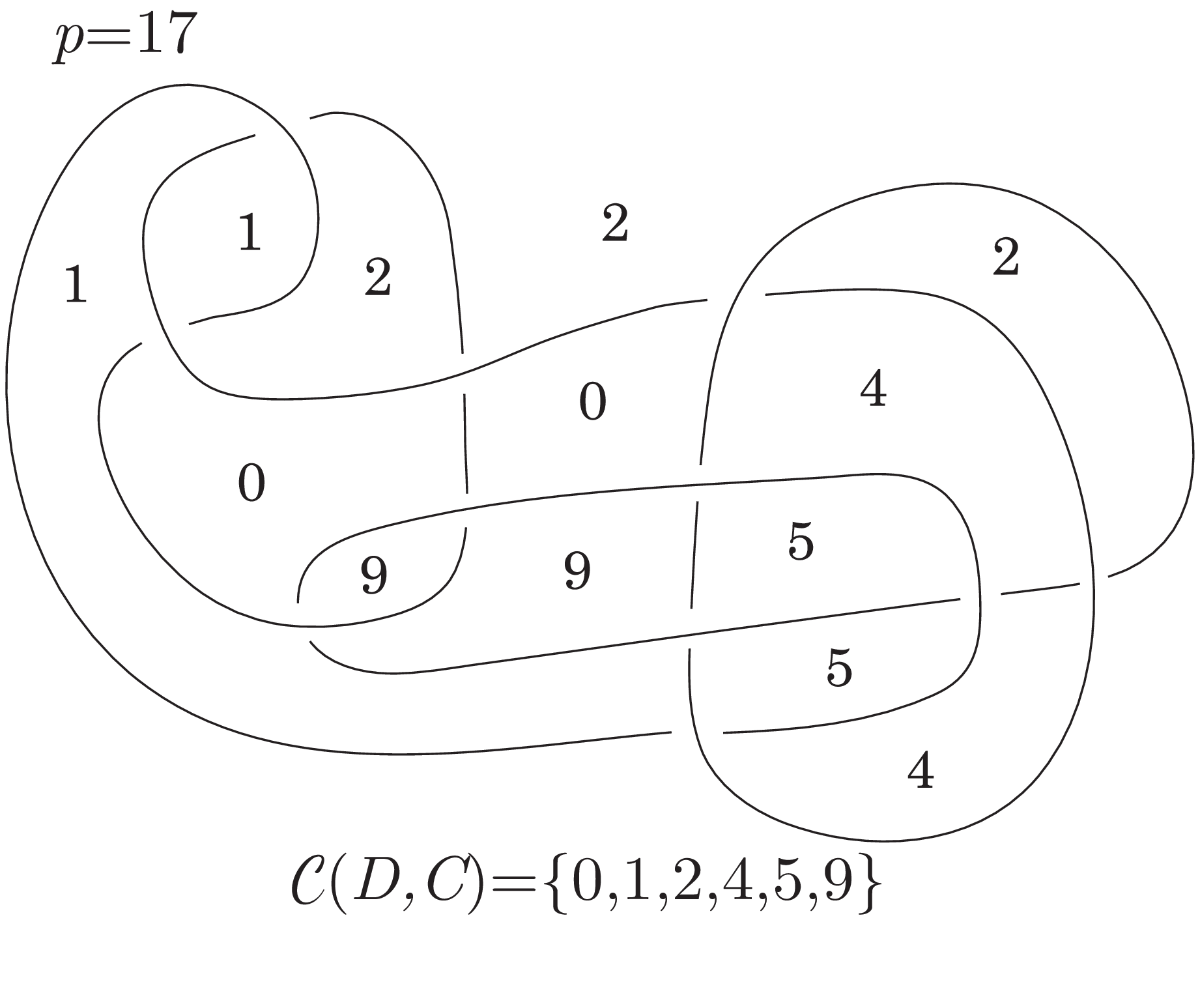}
\caption{} 
\label{p=17diagex}
\end{center}
\end{figure}

\end{remark}

The next questions are open.
\begin{question}
\begin{enumerate}
\item  When $p=17$, does there exist a Dehn $p$-colorable knot $K$ with 
 $$\mincol_p(K) > \lfloor  \log_2 p \rfloor +2~?$$ 
\item When $p\in \{17, 23, 31\}$, 
does there exist a Dehn $p$-colorable knot $K$ with 
 $$\mincol_p(K) = \lfloor  \log_2 p \rfloor +3~?$$ 
\end{enumerate}
\end{question}

\begin{remark}
The cocycle invariant $\Phi_{\theta_p}^{\rm NT}(K)$ for an unoriented knot $K$ can be interpreted by the shadow quandle cocycle invariants of $K$ with the nontrivial shadow $R_p$-colorings and the Mochizuki's 3-cocycle $\theta_p^{\rm M}$ (see \cite{Oshiro20}), where $R_p$ is the dihedral quandle of order $p$.
Here when we calculate the shadow quandle cocycle invariants of $K$, we may give an arbitrary orientation for $K$.
When we denote by $\tilde\Phi_{\theta_p^{\rm M}}^{\rm NT}(K)$ the shadow quandle cocycle invariant corresponding to $\Phi_{\theta_p}^{\rm NT}(K)$, Theorem~\ref{thm:main2} can be also written as the following corollary:
\begin{corollary}
Let $p$ be an odd prime number with $5<p<2^5$ and $p\not =17$.
Let $K$ be a Dehn $p$-colorable knot (i.e., an $R_p$-colorable knot). 

If the shadow quandle cocycle invariant $\tilde\Phi_{\theta_p^{\rm M}}^{\rm NT}(K)$ satisfies $0 \not\in \tilde\Phi_{\theta_p^{\rm M}}^{\rm NT}(K)$, then we have 
\[
\mincol_p(K) \geq \lfloor \log_2 p \rfloor +3.
\]
\end{corollary}
 
\end{remark}

\section{Proof of Theorem~\ref{thm:main3}}

In this section, a crossing of type $\big((a,b),(a,c)\big)$ means a crossing whose weight is $+\big((a,b),(a,c)\big)$ or $-\big((a,b),(a,c)\big)$.

\begin{proof}[Proof of Theorem~\ref{thm:main3}~{\rm (}$p=13${\rm )}]
Let $(D,C)$ be a nontrivially Dehn $13$-colored diagram of a knot $K$.
Assume that $\#\C(D,C)=\lfloor \log_2 13 \rfloor +2=5$. 
Then,  by Lemma~\ref{lemma:colorequiv} and Theorem~\ref{thm:colorcandidate}, we may assume that $\C(D,C)=\{0,1,2,4,7\}$.

For any nontrivially colored crossing $\chi$, the weight of $\chi$ is 
\begin{align*}
&&w_{\chi} =& 
\pm   \big((0,0),(0,1)\big) , \pm  \big((0,0),(0,2)\big) , \pm  \big((0,0),(0,4)\big) \\
&&&  \pm  \big((0,0),(0,7)\big),  \pm  \big((1,1),(1,2)\big), \pm  \big((1,1),(1,4)\big) \\ &&&  \pm  \big((1,1),(1,7)\big), \pm  \big((2,2),(2,4)\big), \pm  \big((2,2),(2,7)\big) \\
&&&   \pm  \big((4,4),(4,7)\big), \pm  \big((0,1),(0,2)\big), \pm  \big((0,2),(0,4)\big) \\ &&&  \pm  \big((0,7),(0,1)\big) , \mbox{ or } \pm  \big((1,4),(1,7)\big). 
\end{align*}
Hence, the sum of the weights of all the nontrivially colored crossings is 
\begin{align*}
&&W(D,C) =& 
s_1   \big((0,0),(0,1)\big) + s_2  \big((0,0),(0,2)\big) + s_3  \big((0,0),(0,4)\big) \\
&&&  s_4  \big((0,0),(0,7)\big) +  s_5  \big((1,1),(1,2)\big) + s_6  \big((1,1),(1,4)\big) \\ &&&  s_7  \big((1,1),(1,7)\big) + s_8  \big((2,2),(2,4)\big) + s_9  \big((2,2),(2,7)\big) \\
&&&   s_{10}  \big((4,4),(4,7)\big) + t_1  \big((0,1),(0,2)\big) + t_2  \big((0,2),(0,4)\big) \\ 
&&&  t_3  \big((0,7),(0,1)\big) + t_4  \big((1,4),(1,7)\big). 
\end{align*}
for some $s_1, \ldots , s_{10}$, $t_1, \ldots , t_4 \in \mathbb Z$, where we note that when we regard $W(D,C)$ as an element of $C_2^{\rm SLB}(\mathbb Z_{13})$,  $s_1, \ldots , s_{10}$ should be dealt as elements of $\mathbb Z_2$. Since
\begin{align*}
0=&\partial_2^{\rm SLB}(W(D,C))\\
=&s_1   \partial_2^{\rm SLB}\Big(\big((0,0),(0,1)\big)\Big) + s_2  \partial_2^{\rm SLB}\Big(\big((0,0),(0,2)\big)\Big) + s_3  \partial_2^{\rm SLB}\Big(\big((0,0),(0,4)\big)\Big) \\
&   s_4  \partial_2^{\rm SLB}\Big(\big((0,0),(0,7)\big)\Big) + s_5  \partial_2^{\rm SLB}\Big(\big((1,1),(1,2)\big)\Big) + s_6  \partial_2^{\rm SLB}\Big(\big((1,1),(1,4)\big)\Big) \\
&   s_7  \partial_2^{\rm SLB}\Big(\big((1,1),(1,7)\big)\Big) + s_8  \partial_2^{\rm SLB}\Big(\big((2,2),(2,4)\big)\Big) + s_9  \partial_2^{\rm SLB}\Big(\big((2,2),(2,7)\big)\Big) \\
&   s_{10}  \partial_2^{\rm SLB}\Big(\big((4,4),(4,7)\big)\Big) + t_1  \partial_2^{\rm SLB}\Big(\big((0,1),(0,2)\big)\Big) + t_2  \partial_2^{\rm SLB}\Big(\big((0,2),(0,4)\big)\Big) \\
&   t_3  \partial_2^{\rm SLB}\Big(\big((0,7),(0,1)\big)\Big)+ t_4  \partial_2^{\rm SLB}\Big(\big((1,4),(1,7)\big)\Big) \\
=&s_1\Big((0,0)-(1,1)\Big) + s_2\Big((0,0)-(2,2)\Big) + s_3\Big((0,0)-(4,4)\Big) \\
&+ s_4\Big((0,0)-(7,7)\Big) + s_5\Big((1,1)-(2,2)\Big) + s_6\Big((1,1)-(4,4)\Big) \\
&+ s_7\Big((1,1)-(7,7)\Big) + s_8\Big((2,2)-(4,4)\Big) + s_9\Big((2,2)-(7,7)\Big) \\
& + s_{10}\Big((4,4)-(7,7)\Big) + t_1\Big((0,1)-(0,2)+(1,1)+(1,2)\Big) \\
&+ t_2\Big((0,2)-(0,4)+(2,2)+(2,4)\Big) \\
&+ t_3\Big((0,7)-(0,1)+(7,7)-(1,7)\Big) \\
&+ t_4\Big((1,4)-(1,7)+(4,4)+(4,7)\Big) \\
=&(s_1+s_2+s_3+s_4)  (0,0)  +(-s_1+s_5+s_6+s_7+t_1)  (1,1)  \\
&+ (-s_2-s_5+s_8+s_9+t_2) (2,2) +(-s_3-s_6-s_8+s_{10}+t_4)  (4,4) \\
&+ (-s_4-s_7-s_9-s_{10}+t_3)  (7,7) 
+(t_1-t_3)  (0,1)  + (-t_1+t_2) (0,2)\\
& -t_2(0,4)  +  t_3(0,7)   +   t_1(1,2)  +  t_4(1,4) +(-t_3-t_4)(1,7)  +  t_2(2,4)   +  t_4(4,7)
\end{align*}
from Lemma~\ref{lem:weightsum2}, 
it follows that 
\begin{align*}
&s_1 \equiv s_5+s_6+s_7 \mbox{ (mod $2$)}, \\
&s_2 \equiv -s_5+s_8+s_9 \mbox{ (mod $2$)}, \\
&s_3 \equiv -s_6-s_8+s_{10} \mbox{ (mod $2$)} \\
&s_4 \equiv -s_7-s_9-s_{10} \mbox{ (mod $2$), and} \\
&t_1=t_2=t_3=t_4=0. 
\end{align*}

This means that there are the same number of crossings with different signs for each type of nontrivially colored crossings in Figure~\ref{p=13crossing}. We note that trivially colored crossings and nontivially colored crossings of type $\big((a,a),(a,b)\big)$ may also exist. 

Suppose that crossings of type $\big((1,4),(1,7)\big)$ exist. Considering how to connect each crossing, we can see that $\{1,4\}$-semiarcs and $\{4,7\}$-semiarcs make a knot component, and other arcs make other components (see Figure~\ref{p=13crossingconnect2}), which contradicts that $K$ is a knot. Next, suppose that crossings of type $\big((1,4),(1,7)\big)$ do not exist and crossings of type $\big((0,7),(0,1)\big)$ exist. Considering how to connect each crossing, we can see that $\{0,7\}$-semiarcs and $\{1,7\}$-semiarcs make a knot component, and other arcs make other components (see Figure~\ref{p=13crossingconnect3}), which contradicts that $K$ is a knot. Similarly, the same holds for cases such that there exist crossings of type $\big((0,1),(0,2)\big)$, but do not exist crossings of type $\big((1,4),(1,7)\big)$ and $\big((0,7),(0,1)\big)$, that there exist crossings of type $\big((0,2),(0,4)\big)$, but do not exist crossings of type $\big((1,4),(1,7)\big)$, $\big((0,7),(0,1)\big)$ and $\big((0,1),(0,2)\big)$, and that there do not exist crossings of type $\big((1,4),(1,7)\big)$, $\big((0,7),(0,1)\big)$, $\big((0,2),(0,4)\big)$ and $\big((0,1),(0,2)\big)$.

Thus, $\sharp \mathcal{C}(D,C) \geq 6$ holds, and consequently $\mincol_{13}(K) \geq \lfloor \log_2 13 \rfloor +3$. 

\begin{figure}[h]
\begin{center}
\includegraphics[width=10cm]{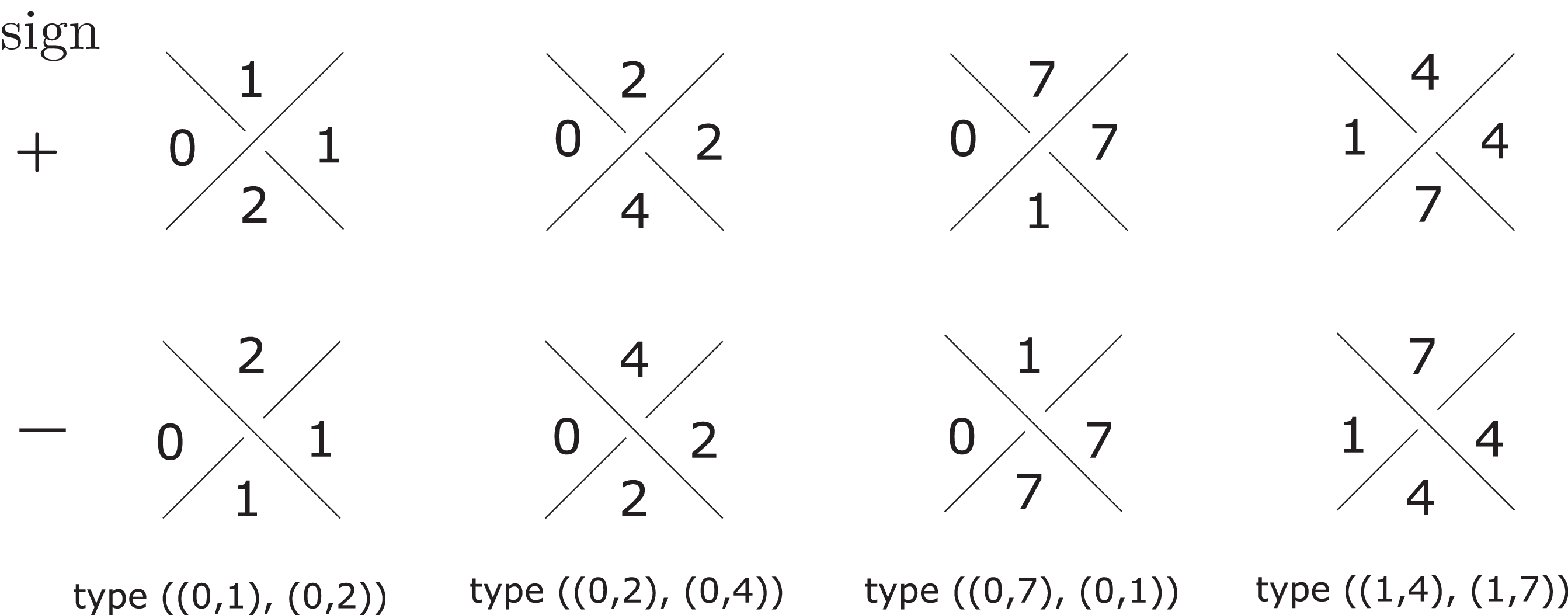}
\caption{} 
\label{p=13crossing}
\end{center}
\end{figure}

\begin{figure}[h]
\begin{center}
\includegraphics[width=8cm]{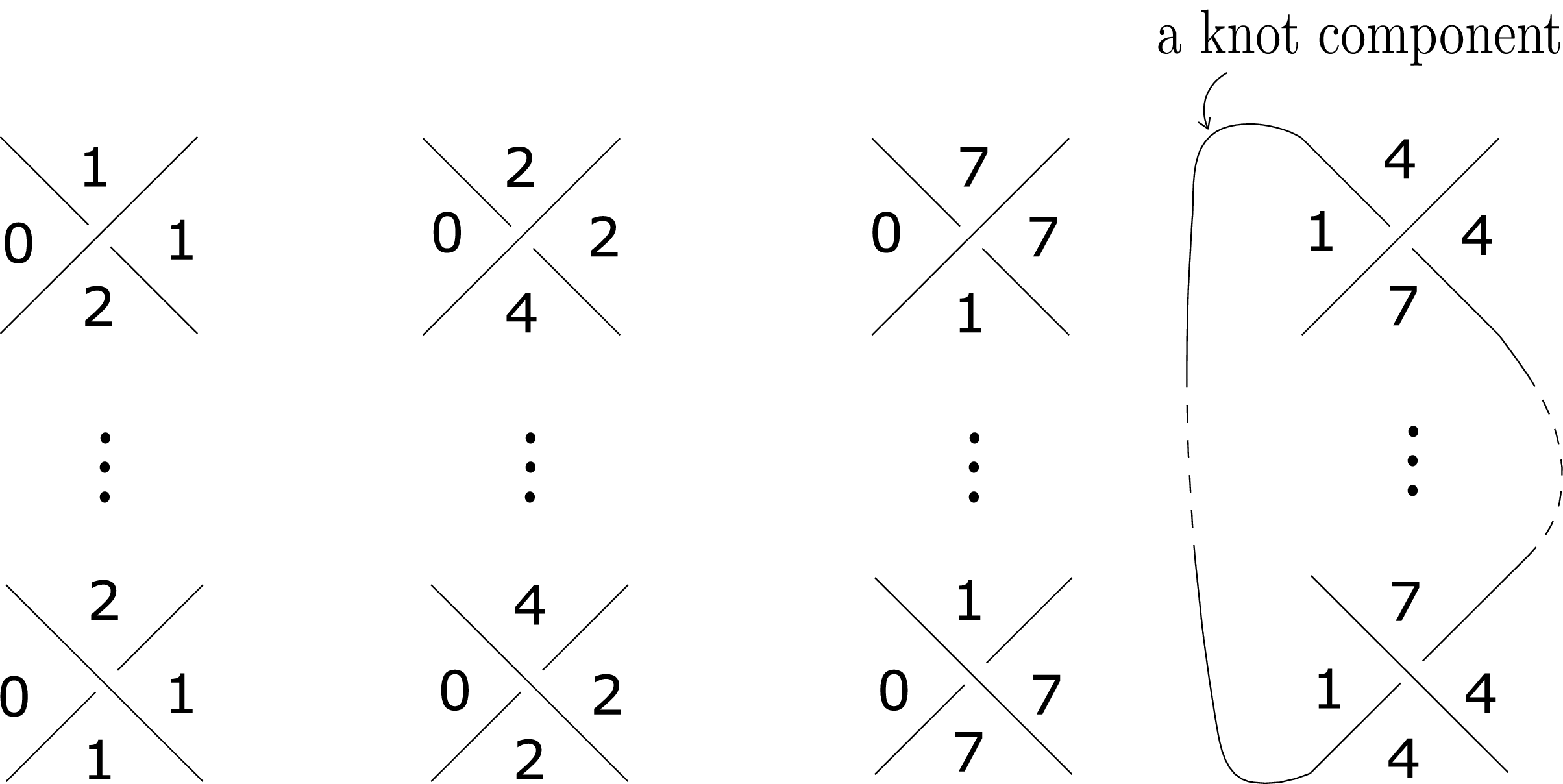}
\caption{} 
\label{p=13crossingconnect2}
\end{center}
\end{figure}

\begin{figure}[h]
\begin{center}
\includegraphics[width=6cm]{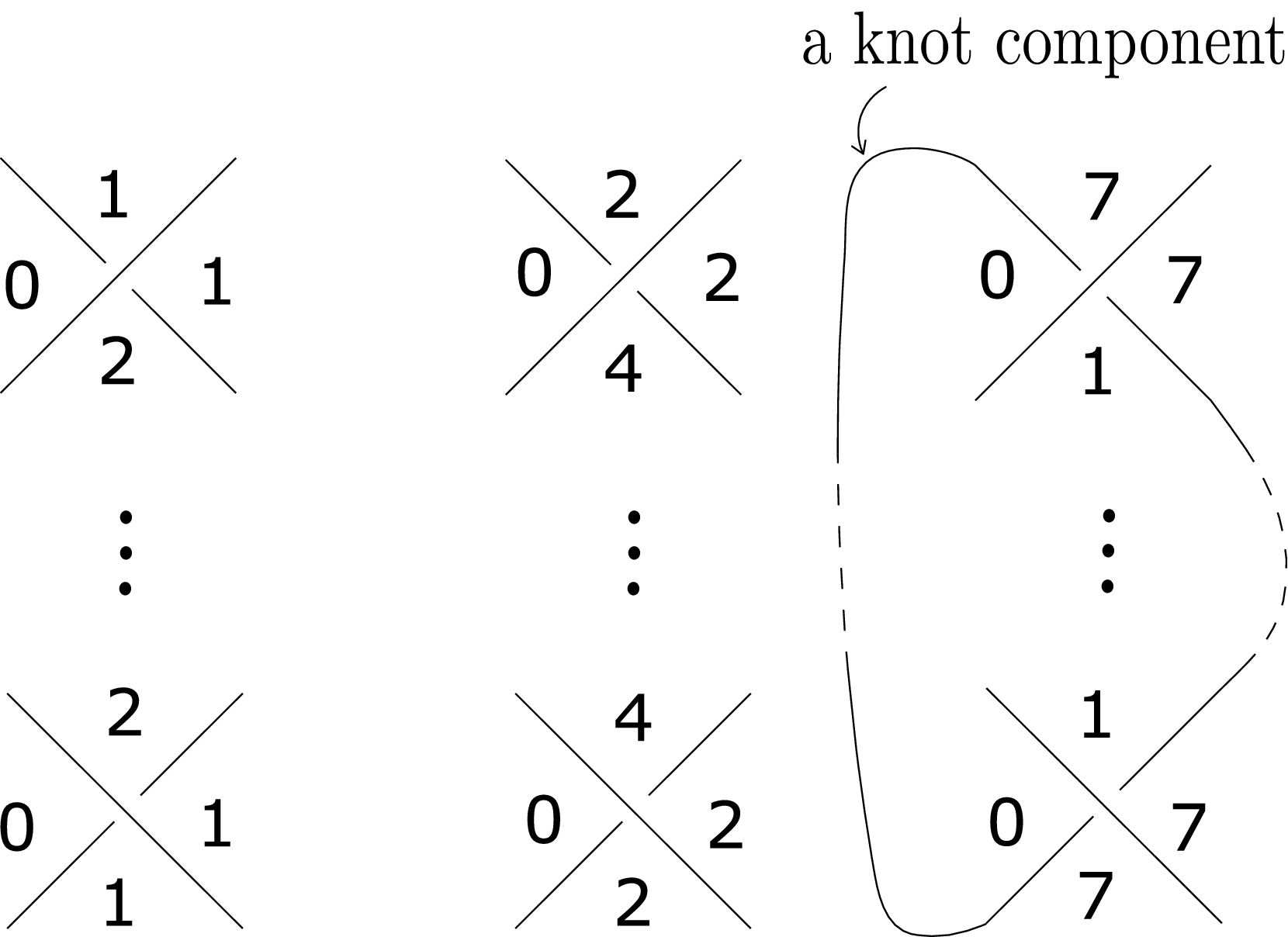}
\caption{} 
\label{p=13crossingconnect3}
\end{center}
\end{figure}

\end{proof}

\begin{proof}[Proof of Theorem~\ref{thm:main3}~{\rm (}$p=29${\rm )}]
Let $(D,C)$ be a nontrivially Dehn $29$-colored diagram of a knot $K$.
Assume that $\#\C(D,C)=\lfloor \log_2 29 \rfloor +2=6$. 
Then,  by Lemma~\ref{lemma:colorequiv} and Theorem~\ref{thm:colorcandidate}, we may assume that $\C(D,C)=\{0,1,2,4,8,15\}$.

Then, the sum of the weights of all the nontrivially colored crossings is
\begin{align*}
&&W(D,C) =& 
\displaystyle \sum_{a,b\in \C(D,C);a<b} s_{a,b}  \big((a,a),(a,b)\big) \\
&&&  + t_1  \big((0,1),(0,2)\big) + t_2  \big((0,2),(0,4)\big) + t_3  \big((0,4),(0,8)\big) \\
&&&  + t_4  \big((0,15),(0,1)\big) + t_5  \big((1,8),(1,15)\big)  
\end{align*}
for some $s_{a,b}$'s, $t_1, \ldots , t_5 \in \mathbb Z$, where we note that when we regard $W(D,C)$ as an element of $C_2^{\rm SLB}(\mathbb Z_{29})$,  $s_{a,b}$'s should be dealt as elements of $\mathbb Z_2$.
Since
\begin{align*}
0=&\partial_2^{\rm SLB}(W(D,C))\\
=&\displaystyle \sum_{a\in \C(D,C)}   S_{a}  (a,a) +(t_1-t_4) (0,1)  +  (-t_1+t_2) (0,2) \\
&+(-t_2+t_3)  (0,4)  -t_3  (0,8)  +t_4 (0,15) \\
&+ t_1  (1,2)  +t_5  (1,8)  + (-t_4-t_5) (1,15) \\
&+t_2  (2,4)  + t_3  (4,8)  +t_5  (8,15) 
\end{align*}
holds for some $S_{a}\in \mathbb Z_2$ from Lemma~\ref{lem:weightsum2}, we obtain
\begin{align*}
&t_1=t_2=t_3=t_4=t_5=0. 
\end{align*}
It means that there are the same number of crossings with different signs for each type of nontrivially colored crossings in Figure~\ref{p=29crossing}. We note that trivially colored crossings and nontivially colored crossings of type $\big((a,a),(a,b)\big)$ may also exist. 

Suppose that crossings of type $\big((1,8),(1,15)\big)$ exist. Considering how to connect each crossings, we can see that $\{1,8\}$-semiarcs and $\{8,15\}$-semiarcs make a knot component, and other arcs make other components (see Figure~\ref{p=29crossingconnect2}), which contradicts that $K$ is a knot. 
Similarly, the same holds for cases such that
there exist crossings of type $\big((0,15),(0,1)\big)$, but do not exist crossings of type $\big((1,8),(1,15)\big)$, 
that there exist crossings of type $\big((0,1),(0,2)\big)$, but do not exist crossings of type $\big((1,8),(1,15)\big)$ and $\big((0,15),(0,1)\big)$, 
that there exist crossings of type $\big((0,2),(0,4)\big)$, but do not exist crossings of type $\big((1,8),(1,15)\big)$, $\big((0,15),(0,1)\big)$ and $\big((0,1),(0,2)\big)$, 
that there exist crossings of type $\big((0,4),(0,8)\big)$, but do not exist crossings of type $\big((1,8),(1,15)\big)$, $\big((0,15),(0,1)\big)$, $\big((0,2),(0,4)\big)$ and $\big((0,1),(0,2)\big)$, and 
that there do not exist crossings of type $\big((1,8),(1,15)\big)$, $\big((0,15),(0,1)\big)$, $\big((0,4),(0,8)\big)$, $\big((0,2),(0,4)\big)$ and $\big((0,1),(0,2)\big)$.

Thus, $\sharp \mathcal{C}(D,C) \geq 7$ holds, and consequently $\mincol_{29}(K) \geq \lfloor \log_2 29 \rfloor +3$. 
\begin{figure}[h]
\begin{center}
\includegraphics[width=10cm]{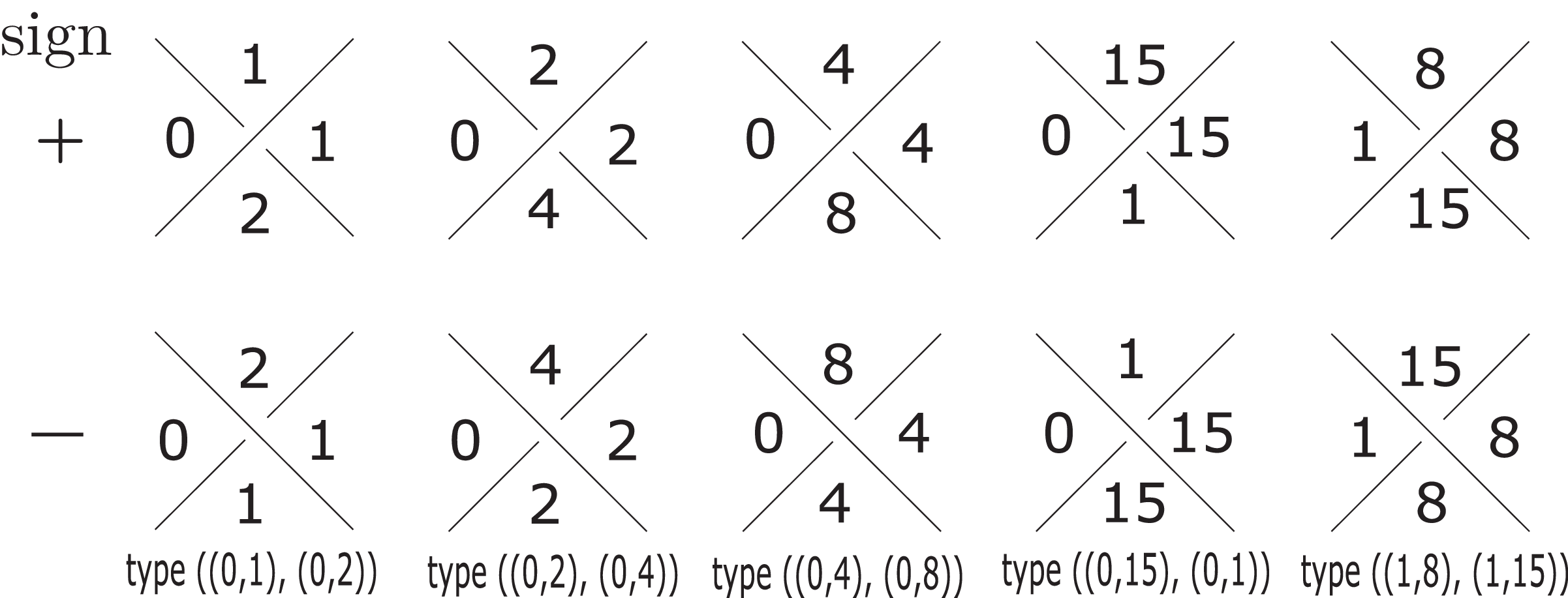}
\caption{} 
\label{p=29crossing}
\end{center}
\end{figure}

\begin{figure}[h]
\begin{center}
\includegraphics[width=8cm]{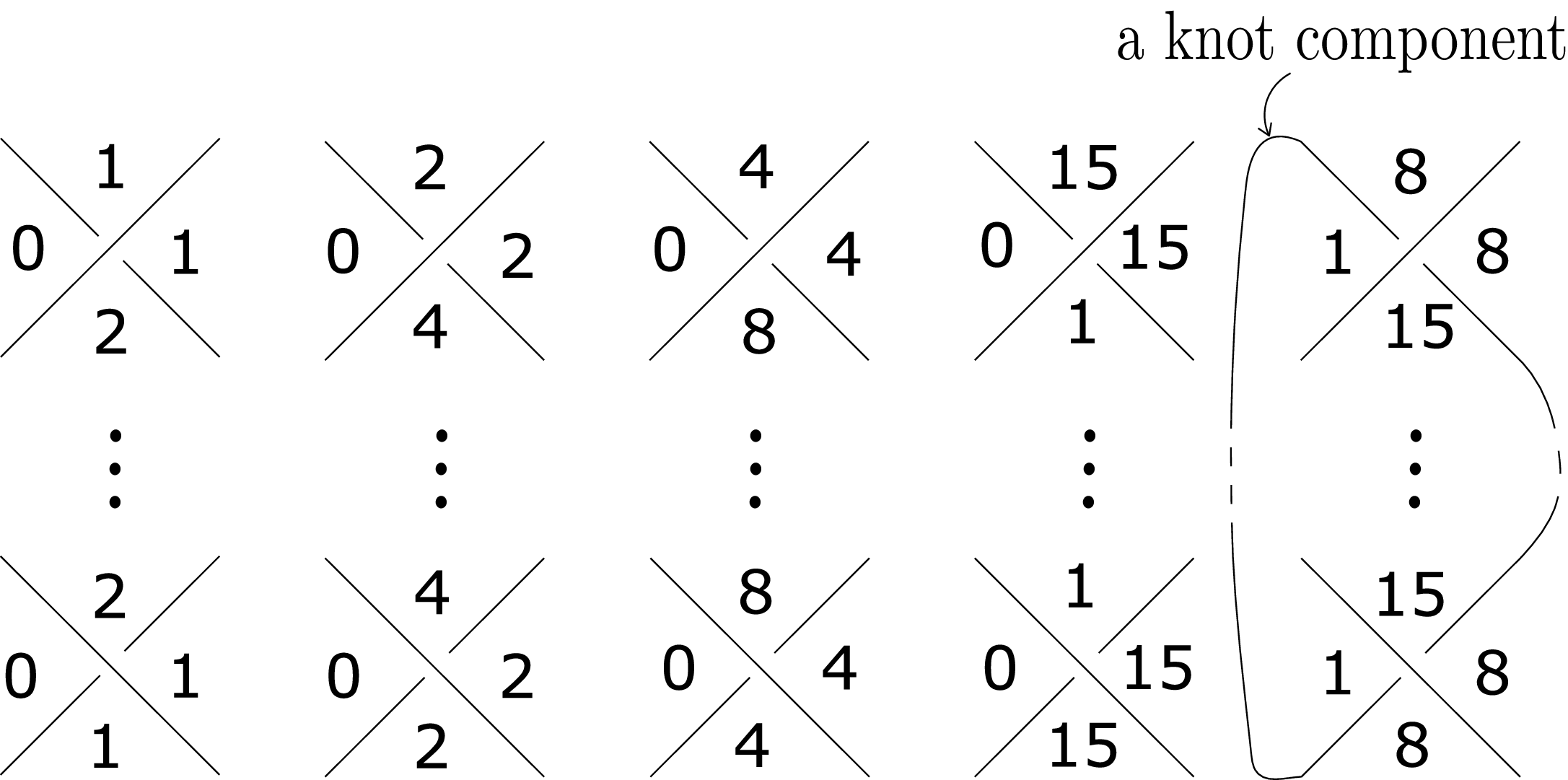}
\caption{} 
\label{p=29crossingconnect2}
\end{center}
\end{figure}

\end{proof}

\section{Proof of Theorem~\ref{thm:main2}}
\begin{proof}[Proof of Theorem~\ref{thm:main2}~{\rm (}$p=7${\rm )}]
Let $(D,C)$ be a nontrivially Dehn $7$-colored diagram of a knot $K$ with $0 \not\in \Phi_{\theta_7}^{\rm NT}(K)$.
Assume that $\#\C(D,C)=4$. 
Then, by Lemma~\ref{lemma:colorequiv} and Theorem~\ref{thm:colorcandidate}, we may assume that $\C(D,C)=\{0,1,2,4\}$.

Then, the sum of the weights of all the nontrivially colored crossings are 
\begin{align*}
&&W(D,C) =& s_1 \big((0,0),(0,1)\big) + s_2 \big((0,0),(0,2)\big)+ s_3 \big((0,0),(0,4)\big)\\
&&& + s_4 \big((1,1),(1,2)\big) + s_5 \big((1,1),(1,4)\big)+ s_6  \big((2,2),(2,4)\big)\\
&&&  + t_1  \big((0,1),(0,2)\big) + t_2  \big((0,2),(0,4)\big) + t_3  \big((0,4),(0,1)\big) 
\end{align*}
for some $s_1 ,\ldots , s_6, t_1, \ldots , t_3 \in \mathbb Z$, where we note that when we regard $W(D,C)$ as an element of $C_2^{\rm SLB}(\mathbb Z_7)$,  $s_1 ,\ldots , s_6$ should be dealt as elements of $\mathbb Z_2$.
From Lemma~\ref{lem:weightsum2}, since we have $\partial_2^{\rm SLB}(W(D,C)) =0$, that is, 
\begin{align*}
0=& s_1 \partial_2^{\rm SLB}\Big(\big((0,0),(0,1)\big)\Big) + s_2 \partial_2^{\rm SLB}\Big(\big((0,0),(0,2)\big)\Big)+ s_3 \partial_2^{\rm SLB}\Big(\big((0,0),(0,4)\big)\Big)\\
& + s_4 \partial_2^{\rm SLB}\Big(\big((1,1),(1,2)\big)\Big) + s_5 \partial_2^{\rm SLB}\Big(\big((1,1),(1,4)\big)\Big)+ s_6  \partial_2^{\rm SLB}\Big(\big((2,2),(2,4)\big)\Big)\\
&  + t_1  \partial_2^{\rm SLB}\Big(\big((0,1),(0,2)\big)\Big) + t_2  \partial_2^{\rm SLB}\Big(\big((0,2),(0,4)\big)\Big) + t_3  \partial_2^{\rm SLB}\Big(\big((0,4),(0,1)\big)\Big) \\
=& s_1 \Big( (0,0) - (1,1) \Big) + s_2\Big( (0,0) - (2,2) \Big)+ s_3 \Big( (0,0) - (4,4) \Big)\\
& + s_4 \Big( (1,1) - (2,2) \Big) + s_5 \Big( (1,1) - (4,4) \Big)+ s_6  \Big( (2,2) - (4,4) \Big)\\
&  + t_1  \Big(- (0,2) + (1,1) + (0,1) + (1,2) \Big) + t_2   \Big(- (0,4) + (2,2) + (0,2) + (2,4) \Big)\\
& + t_3   \Big(- (0,1) + (4,4) + (0,4) - (1,4) \Big) \\
=&  (s_1+s_2+s_3)   (0,0) + (-s_1+s_4+s_5+t_1)  (1,1)   + (-s_2-s_4 + s_6+t_2)  (2,2) \\
&+  (-s_3  -s_5 -s_6+ t_3) (4,4)  +(t_1-t_3)  (0,1)   +  (-t_1+t_2)  (0,2)  \\
&+(-t_2+t_3)   (0,4)   +t_1  (1,2)   - t_3  (1,4)   +  t_2 (2,4) ,
\end{align*}
it follows that 
\begin{align*}
&s_1 \equiv s_4+s_5 \pmod{2},\\
&s_2\equiv  s_4+s_6 \pmod{2}, \\
&s_3\equiv  s_5+s_6 \pmod{2},  \mbox{ and}\\
&t_1=t_2=t_3=0. \\
\end{align*}

Therefore, we have 
\begin{align*}
\theta_7(W(D,C))=
& (s_4+s_5) \theta_7 \Big(\big((0,0),(0,1)\big)\Big) + (s_4+s_6) \theta_7 \Big(\big((0,0),(0,2)\big)\Big)\\
&+ (s_5+s_6) \theta_7 \Big( \big((0,0),(0,4)\big)\Big)+ s_4 \theta_7 \Big(\big((1,1),(1,2)\big) \Big)\\
&+ s_5 \theta_7 \Big(\big((1,1),(1,4)\big)\Big)+ s_6  \theta_7 \Big(\big((2,2),(2,4)\big)\Big)\\
=& (s_4+s_5) 0 + (s_4+s_6) 0+ (s_5+s_6) 0+ s_4 0+ s_5 0+ s_6 0 \\
=& 0,
\end{align*}
which contradicts $0 \not\in \Phi_{\theta_7}^{\rm NT}(K)$. Thus $\# \C(D,C) \geq 5$ holds, and consequently  $\mincol_7(K)\geq 5$.
\end{proof}
Similar arguments apply to the other cases, and hence, we will briefly sketch the proofs of the cases when $p\in \{11,13, 19, 23, 29, 31\}$.

\begin{proof}[Proof of Theorem~\ref{thm:main2}~{\rm (}$p=11${\rm )}]
Let $(D,C)$ be a nontrivially Dehn $11$-colored diagram of a knot $K$ with $0 \not\in \Phi_{\theta_{11}}^{\rm NT}(K)$.
Assume that $\#\C(D,C)=5$. 
Then, by Lemma~\ref{lemma:colorequiv} and Theorem~\ref{thm:colorcandidate},  
we may assume that $\C(D,C)=\{0,1,2,3,6\}$ or $\{0,1,2,4,7\}$.

For the case that $\C(D,C)=\{0,1,2,3,6\}$, the sum of the weights of all the nontrivially colored crossings is
\begin{align*}
&&W(D,C) =& 
\displaystyle \sum_{a,b\in \C(D,C);a<b} s_{a,b}  \big((a,a),(a,b)\big) \\
&&&  + t_1  \big((0,1),(0,2)\big) + t_2  \big((0,1),(0,3)\big) + t_3  \big((0,2),(0,3)\big) \\
&&&  + t_4  \big((0,3),(0,6)\big)+ t_5  \big((0,6),(0,1)\big)+ t_6  \big((1,2),(1,3)\big) 
\end{align*}
for some $s_{a,b}$'s, $t_1, \ldots , t_6 \in \mathbb Z$, where we note that when we regard $W(D,C)$ as an element of $C_2^{\rm SLB}(\mathbb Z_{11})$,  $s_{a,b}$'s should be dealt as elements of $\mathbb Z_2$.
Since
\begin{align*}
0=&\partial_2^{\rm SLB}(W(D,C))\\
=&\displaystyle \sum_{a\in \C(D,C)}   S_{a}  (a,a) +(t_1+t_2-t_5) (0,1)  +  (-t_1+t_3) (0,2) \\
&+(-t_2-t_3+t_4)  (0,3)  +(-t_4+t_5) (0,6)  + (t_1+t_2-t_3+t_6) (1,2) \\
&+(t_3-t_6)(1,3)  -  t_5(1,6)   +  (t_2+t_6)(2,3)  +  t_4(3,6) 
\end{align*}
holds for some $S_a \in \mathbb{Z}_2$ from Lemma~\ref{lem:weightsum2}, we obtain
\begin{align*}
&t_1=-t_2=t_3=t_6,\mbox{ and  } t_4=t_5=0. 
\end{align*}
Therefore, we have 
\begin{align*}
\theta_{11}(W(D,C))=
& \displaystyle \sum_{a,b\in \C(D,C);a<b} s_{a,b} \theta_{11} \Big( \big((a,a),(a,b)\big) \Big)\\
&+ t_1 \theta_{11} \Big(\big((0,1),(0,2)\big)\Big)- t_1 \theta_{11} \Big(\big((0,1),(0,3)\big)\Big)\\
&+ t_1 \theta_{11} \Big(\big((0,2),(0,3)\big)\Big)+ t_1 \theta_{11} \Big(\big((1,2),(1,3)\big)\Big)\\
=&9t_1 - 9t_1 + 10t_1 + t_1\\
=& 0,
\end{align*}
which contradicts $0 \not\in \Phi_{\theta_{11}}^{\rm NT}(K)$. 

For the case that $\C(D,C)=\{0,1,2,4,7\}$, the sum of the weights of all the nontrivially colored crossings is
\begin{align*}
&&W(D,C) =& 
\displaystyle \sum_{a,b\in \C(D,C);a<b} s_{a,b}  \big((a,a),(a,b)\big) \\
&&&  + t_1  \big((0,1),(0,2)\big) + t_2  \big((0,2),(0,4)\big) + t_3  \big((0,4),(0,0)\big) \\
&&&  + t_4  \big((1,4),(1,7)\big) + t_5  \big((1,7),(1,2)\big) 
\end{align*}
for some $s_{a,b}$'s, $t_1, \ldots , t_5 \in \mathbb Z$, where we note that when we regard $W(D,C)$ as an element of $C_2^{\rm SLB}(\mathbb Z_{11})$,  $s_{a,b}$ should be dealt as elements of $\mathbb Z_2$.
Since
\begin{align*}
0=&\partial_2^{\rm SLB}(W(D,C))\\
=&\displaystyle \sum_{a\in \C(D,C)}   S_{a}  (a,a)   +t_1 (0,1)  +  (-t_1+t_2) (0,2)  \\
&+(-t_2+t_3)  (0,4)  - t_3 (0,7) +  (t_1-t_5)(1,2) +t_4  (1,4)\\
&+(-t_4+t_5) (1,7)    +  t_2(2,4) -t_5  (2,7)   + (t_3+t_4) (4,7)  
\end{align*}
holds for some $S_a \in \mathbb{Z}_2$ from Lemma~\ref{lem:weightsum2}, we obtain
\begin{align*}
&t_1=t_2=t_3=t_4=t_5=0. 
\end{align*}
Therefore, we have 
\begin{align*}
\theta_{11}(W(D,C))=
& \displaystyle \sum_{a,b\in \C(D,C);a<b} s_{a,b} \theta_{11} \Big( \big((a,a),(a,b)\big) \Big)\\
=& 0,
\end{align*}
which contradicts $0 \not\in \Phi_{\theta_{11}}^{\rm NT}(K)$. 

From the above discussion, we conclude that $\# \C(D,C) \geq 6$, and consequently  $\mincol_{11}(K)\geq 6$.
\end{proof}

\begin{proof}[Proof of Theorem~\ref{thm:main2}~{\rm (}$p=13${\rm )}]
When $p=13$, it is clear from Theorem~\ref{thm:main3} that $\mincol_{13}(K)\geq 6$ holds.
\end{proof}

\begin{proof}[Proof of Theorem~\ref{thm:main2}~{\rm (}$p=19${\rm )}]
Let $(D,C)$ be a nontrivially Dehn $19$-colored diagram of a knot $K$ with $0 \not\in \Phi_{\theta_{19}}^{\rm NT}(K)$.
Assume that $\#\C(D,C)=6$. 
Then, by Lemma~\ref{lemma:colorequiv} and Theorem~\ref{thm:colorcandidate}, 
we may assume that $\C(D,C)=\{0,1,2,3,5,10\}$, $\{0,1,2,3,6,10\}$, $\{0,1,2,3,6,11\}$, $\{0,1,2,3,6,12\}$, $\{0,1,2,3,6,13\}$, $\{0,1,2,3,6,14\}$, $\{0,1,2,3,7,12\}$, $\{0,1,2,4,5,10\}$, $\{0,1,2,4,5,14\}$, $\{0,1,2,4,7,12\}$ or $\{0,1,2,4,7,15\}$.

For the case that $\C(D,C)=\{0,1,2,3,5,10\}$, the sum of the weights of all the nontrivially colored crossings is
\begin{align*}
&&W(D,C) =& 
\displaystyle \sum_{a,b\in \C(D,C);a<b} s_{a,b}  \big((a,a),(a,b)\big) \\
&&&  + t_1  \big((0,1),(0,2)\big) + t_2  \big((0,1),(0,3)\big) + t_3  \big((0,2),(0,3)\big) \\
&&&  + t_4  \big((0,2),(0,5)\big) + t_5  \big((0,3),(0,5)\big) + t_6  \big((0,5),(0,10)\big) \\
&&&  + t_7  \big((0,10),(0,1)\big) + t_8  \big((1,2),(1,3)\big) + t_9  \big((1,3),(1,5)\big) 
\end{align*}
for some $s_{a,b}$'s, $t_1, \ldots , t_9 \in \mathbb Z$, where we note that when we regard $W(D,C)$ as an element of $C_2^{\rm SLB}(\mathbb Z_{19})$,  $s_{a,b}$'s should be dealt as elements of $\mathbb Z_2$.
Since
\begin{align*}
0=&\partial_2^{\rm SLB}(W(D,C))\\
=&\displaystyle \sum_{a\in \C(D,C)}   S_{a}   (a,a)  +(t_1+t_2-t_7)  (0,1)   +  (-t_1+t_3+t_4)  (0,2)  \\
&+(-t_2-t_3+t_5)   (0,3)   +(-t_4-t_5+t_6)   (0,5)   +(-t_6+t_7)  (0,10)  \\
&+(t_1+t_2-t_3+t_8)   (1,2)   +(t_3-t_8+t_9)   (1,3)   -t_9  (1,5)  -t_7   (1,10) \\
&  +(t_2+t_4-t_5+t_8)   (2,3)   +t_5  (2,5) +(t_4+t_9)   (3,5)   +t_6   (5,10)  
\end{align*}
holds for some $S_a \in \mathbb{Z}_2$ from Lemma~\ref{lem:weightsum2}, we obtain
\begin{align*}
&t_1=-t_2=t_3=t_8,  \mbox{ and } t_4=t_5=t_6=t_7=t_9=0. 
\end{align*}
Therefore, we have 
\begin{align*}
\theta_{19}(W(D,C))=
& \displaystyle \sum_{a,b\in \C(D,C);a<b} s_{a,b} \theta_{19} \Big( \big((a,a),(a,b)\big) \Big)\\
&+ t_1 \theta_{19} \Big(\big((0,1),(0,2)\big)\Big)- t_1 \theta_{19} \Big(\big((0,1),(0,3)\big)\Big)\\
&+ t_1 \theta_{19} \Big(\big((0,2),(0,3)\big)\Big)+ t_1 \theta_{19} \Big(\big((1,2),(1,3)\big)\Big)\\
=&15t_1 - 10t_1 + 4t_1 + 10t_1\\
=& 0,
\end{align*}
which contradicts $0 \not\in \Phi_{\theta_{19}}^{\rm NT}(K)$. 

For the case that $\C(D,C)=\{0,1,2,3,6,10\}$, the sum of the weights of all the nontrivially colored crossings is
\begin{align*}
&&W(D,C) =& 
\displaystyle \sum_{a,b\in \C(D,C);a<b} s_{a,b}  \big((a,a),(a,b)\big)+ t_1  \big((0,1),(0,2)\big)  \\
&&&  + t_2  \big((0,1),(0,3)\big) + t_3  \big((0,2),(0,3)\big) + t_4  \big((0,3),(0,6)\big)\\
&&&   + t_5  \big((0,10),(0,1)\big) + t_6  \big((1,2),(1,3)\big) + t_7  \big((2,6),(2,10)\big)
\end{align*}
for some $s_{a,b}$'s, $t_1, \ldots , t_7 \in \mathbb Z$, where we note that when we regard $W(D,C)$ as an element of $C_2^{\rm SLB}(\mathbb Z_{19})$,  $s_{a,b}$'s should be dealt as elements of $\mathbb Z_2$.
Since
\begin{align*}
0=&\partial_2^{\rm SLB}(W(D,C))\\
=&\displaystyle \sum_{a\in \C(D,C)}   S_{a}   (a,a)  +(t_1+t_2-t_5)  (0,1)   +  (-t_1+t_3)  (0,2)  +(-t_2-t_3+t_4)   (0,3)\\
&   -t_4   (0,6)   +t_5  (0,10)  +(t_1+t_2-t_3+t_6)   (1,2)   +(t_3-t_6)   (1,3) -t_5  (1,10)  \\
&  +(t_2+t_6)   (2,3)   +t_7   (2,6)   -t_7  (2,10)  +t_4   (3,6)   +t_7   (6,10)  
\end{align*}
holds for some $S_a \in \mathbb{Z}_2$ from Lemma~\ref{lem:weightsum2}, we obtain
\begin{align*}
&t_1=-t_2=t_3=t_6, \mbox{ and }t_4=t_5=t_7=0. 
\end{align*}
Therefore, we have 
\begin{align*}
\theta_{19}(W(D,C))=
& \displaystyle \sum_{a,b\in \C(D,C);a<b} s_{a,b} \theta_{19} \Big( \big((a,a),(a,b)\big) \Big)\\
&+ t_1 \theta_{19} \Big(\big((0,1),(0,2)\big)\Big)- t_1 \theta_{19} \Big(\big((0,1),(0,3)\big)\Big)\\
&+ t_1 \theta_{19} \Big(\big((0,2),(0,3)\big)\Big)+ t_1 \theta_{19} \Big(\big((1,2),(1,3)\big)\Big)\\
=&15t_1 - 10t_1 + 4t_1 + 10t_1\\
=& 0,
\end{align*}
which contradicts $0 \not\in \Phi_{\theta_{19}}^{\rm NT}(K)$.

For the case that $\C(D,C)=\{0,1,2,3,6,11\}$, the sum of the weights of all the nontrivially colored crossings is
\begin{align*}
&&W(D,C) =& 
\displaystyle \sum_{a,b\in \C(D,C);a<b} s_{a,b}  \big((a,a),(a,b)\big) \\
&&&  + t_1  \big((0,1),(0,2)\big) + t_2  \big((0,1),(0,3)\big) + t_3  \big((0,2),(0,3)\big) \\
&&&  + t_4  \big((0,3),(0,6)\big) + t_5  \big((0,11),(0,3)\big) + t_6  \big((1,2),(1,3)\big) \\
&&&  + t_7  \big((1,6),(1,11)\big) + t_8  \big((1,11),(1,2)\big) 
\end{align*}
for some $s_{a,b}$'s, $t_1, \ldots , t_8 \in \mathbb Z$, where we note that when we regard $W(D,C)$ as an element of $C_2^{\rm SLB}(\mathbb Z_{19})$,  $s_{a,b}$'s should be dealt as elements of $\mathbb Z_2$.
Since
\begin{align*}
0=&\partial_2^{\rm SLB}(W(D,C))\\
=&\displaystyle \sum_{a\in \C(D,C)}   S_{a}   (a,a)  +(t_1+t_2)  (0,1)   +  (-t_1+t_3)  (0,2)  \\
&+(-t_2-t_3+t_4-t_5)   (0,3)   -t_4   (0,6)   +t_5  (0,11)  \\
&+(t_1+t_2-t_3+t_6-t_8)   (1,2)   +(t_3-t_6)   (1,3)   + t_7  (1,6)  \\
&+(-t_7+t_8)   (1,11)   +(t_2+t_6)   (2,3)   -t_8  (2,11)  \\
&+t_4   (3,6)   -t_5   (3,11)   +t_7   (6,11)  
\end{align*}
holds for some $S_a \in \mathbb{Z}_2$ from Lemma~\ref{lem:weightsum2}, we obtain
\begin{align*}
&t_1=-t_2=t_3=t_6, \mbox{ and  } t_4=t_5=t_7=t_8=0.
\end{align*}
Therefore, we have 
\begin{align*}
\theta_{19}(W(D,C))=
& \displaystyle \sum_{a,b\in \C(D,C);a<b} s_{a,b} \theta_{19} \Big( \big((a,a),(a,b)\big) \Big)\\
&+ t_1 \theta_{19} \Big(\big((0,1),(0,2)\big)\Big)- t_1 \theta_{19} \Big(\big((0,1),(0,3)\big)\Big)\\
&+ t_1 \theta_{19} \Big(\big((0,2),(0,3)\big)\Big)+ t_1 \theta_{19} \Big(\big((1,2),(1,3)\big)\Big)\\
=&15t_1 - 10t_1 + 4t_1 + 10t_1\\
=& 0,
\end{align*}
which contradicts $0 \not\in \Phi_{\theta_{19}}^{\rm NT}(K)$.

For the case that $\C(D,C)=\{0,1,2,3,6,12\}$, the sum of the weights of all the nontrivially colored crossings is
\begin{align*}
&&W(D,C) =& 
\displaystyle \sum_{a,b\in \C(D,C);a<b} s_{a,b}  \big((a,a),(a,b)\big) + t_1  \big((0,1),(0,2)\big) \\
&&&  + t_2  \big((0,1),(0,3)\big) + t_3  \big((0,2),(0,3)\big) + t_4  \big((0,3),(0,6)\big) \\
&&&  + t_5  \big((0,6),(0,12)\big) + t_6  \big((1,2),(1,3)\big) + t_7  \big((2,12),(2,3)\big)
\end{align*}
for some $s_{a,b}$'s, $t_1, \ldots , t_7 \in \mathbb Z$, where we note that when we regard $W(D,C)$ as an element of $C_2^{\rm SLB}(\mathbb Z_{19})$,  $s_{a,b}$'s should be dealt as elements of $\mathbb Z_2$.
Since
\begin{align*}
0=&\partial_2^{\rm SLB}(W(D,C))\\
=&\displaystyle \sum_{a\in \C(D,C)}   S_{a}   (a,a)  +(t_1+t_2)  (0,1)   +  (-t_1+t_3)  (0,2)  \\
&+(-t_2-t_3+t_4)   (0,3)   +(-t_4+t_5)   (0,6)   -t_5  (0,12)  \\
&+(t_1+t_2-t_3+t_6)   (1,2)   +(t_3-t_6)   (1,3)   + (t_2+t_6-t_7)   (2,3)  \\
&+t_7   (2,12)   +t_4   (3,6)   -t_7   (3,12)  +t_5   (6,12) 
\end{align*}
holds for some $S_a \in \mathbb{Z}_2$ from Lemma~\ref{lem:weightsum2}, we obtain
\begin{align*}
&t_1=-t_2=t_3=t_6, \mbox{ and  } t_4=t_5=t_7=0.
\end{align*}
Therefore, we have 
\begin{align*}
\theta_{19}(W(D,C))=
& \displaystyle \sum_{a,b\in \C(D,C);a<b} s_{a,b} \theta_{19} \Big( \big((a,a),(a,b)\big) \Big)\\
&+ t_1 \theta_{19} \Big(\big((0,1),(0,2)\big)\Big)- t_1 \theta_{19} \Big(\big((0,1),(0,3)\big)\Big)\\
&+ t_1 \theta_{19} \Big(\big((0,2),(0,3)\big)\Big)+ t_1 \theta_{19} \Big(\big((1,2),(1,3)\big)\Big)\\
=&15t_1 - 10t_1 + 4t_1 + 10t_1\\
=& 0,
\end{align*}
which contradicts $0 \not\in \Phi_{\theta_{19}}^{\rm NT}(K)$.

For the case that $\C(D,C)=\{0,1,2,3,6,13\}$, the sum of the weights of all the nontrivially colored crossings is
\begin{align*}
&&W(D,C) =& 
\displaystyle \sum_{a,b\in \C(D,C);a<b} s_{a,b}  \big((a,a),(a,b)\big)+ t_1  \big((0,1),(0,2)\big)  \\
&&&  + t_2  \big((0,1),(0,3)\big) + t_3  \big((0,2),(0,3)\big) + t_4  \big((0,3),(0,6)\big) \\
&&&  + t_5  \big((0,6),(0,0)\big) + t_6  \big((1,2),(1,3)\big) + t_7  \big((1,13),(1,6)\big)
\end{align*}
for some $s_{a,b}$'s, $t_1, \ldots , t_7 \in \mathbb Z$, where we note that when we regard $W(D,C)$ as an element of $C_2^{\rm SLB}(\mathbb Z_{19})$,  $s_{a,b}$'s should be dealt as elements of $\mathbb Z_2$.
Since
\begin{align*}
0=&\partial_2^{\rm SLB}(W(D,C))\\
=&\displaystyle \sum_{a\in \C(D,C)}   S_{a}   (a,a)  +(t_1+t_2)  (0,1)   +  (-t_1+t_3)  (0,2)  \\
&+(-t_2-t_3+t_4)   (0,3)   +(-t_4+t_5)   (0,6)   -t_5  (0,13)  \\
&+(t_1+t_2-t_3+t_6)   (1,2)   +(t_3-t_6)   (1,3)   -t_7  (1,6)   \\
&+t_7  (1,13)  + (t_2+t_6)   (2,3)  +t_4   (3,6)   +(t_5-t_7)   (6,13)  
\end{align*}
holds for some $S_a \in \mathbb{Z}_2$ from Lemma~\ref{lem:weightsum2}, we obtain
\begin{align*}
&t_1=-t_2=t_3=t_6, \mbox{ and  }t_4=t_5=t_7=0.
\end{align*}
Therefore, we have 
\begin{align*}
\theta_{19}(W(D,C))=
& \displaystyle \sum_{a,b\in \C(D,C);a<b} s_{a,b} \theta_{19} \Big( \big((a,a),(a,b)\big) \Big)\\
&+ t_1 \theta_{19} \Big(\big((0,1),(0,2)\big)\Big)- t_1 \theta_{19} \Big(\big((0,1),(0,3)\big)\Big)\\
&+ t_1 \theta_{19} \Big(\big((0,2),(0,3)\big)\Big)+ t_1 \theta_{19} \Big(\big((1,2),(1,3)\big)\Big)\\
=&15t_1 - 10t_1 + 4t_1 + 10t_1\\
=& 0,
\end{align*}
which contradicts $0 \not\in \Phi_{\theta_{19}}^{\rm NT}(K)$. 

For the case that $\C(D,C)=\{0,1,2,3,6,14\}$, the sum of the weights of all the nontrivially colored crossings is
\begin{align*}
&&W(D,C) =& 
\displaystyle \sum_{a,b\in \C(D,C);a<b} s_{a,b}  \big((a,a),(a,b)\big) \\
&&&  + t_1  \big((0,1),(0,2)\big) + t_2  \big((0,1),(0,3)\big) + t_3  \big((0,2),(0,3)\big) \\
&&&  + t_4  \big((0,3),(0,6)\big) + t_5  \big((0,6),(0,1)\big) + t_6  \big((0,14),(0,1)\big) \\
&&&  + t_7  \big((1,2),(1,3)\big) + t_8  \big((3,14),(3,6)\big)
\end{align*}
for some $s_{a,b}$'s, $t_1, \ldots , t_8 \in \mathbb Z$, where we note that when we regard $W(D,C)$ as an element of $C_2^{\rm SLB}(\mathbb Z_{19})$,  $s_{a,b}$'s should be dealt as elements of $\mathbb Z_2$.
Since
\begin{align*}
0=&\partial_2^{\rm SLB}(W(D,C))\\
=&\displaystyle \sum_{a\in \C(D,C)}   S_{a}   (a,a)  +(t_1+t_2-t_5-t_6)  (0,1)   +  (-t_1+t_3)  (0,2)  \\
&+(-t_2-t_3+t_4)   (0,3)   +(-t_4+t_5)   (0,6)   +t_6  (0,14)  \\
&+(t_1+t_2-t_3+t_7)   (1,2)   +(t_3-t_7)   (1,3)   -t_6  (1,6)   \\
&-t_5  (1,14)  + (t_2+t_7)   (2,3)  +(t_4-t_8)   (3,6)   +t_8   (3,14)  \\
&+(t_5-t_6-t_8)   (6,14)  
\end{align*}
holds for some $S_a \in \mathbb{Z}_2$ from Lemma~\ref{lem:weightsum2}, we obtain
\begin{align*}
&t_1=-t_2=t_3=t_7, \mbox{ and  } t_4=t_5=t_6=t_8=0.  
\end{align*}
Therefore, we have 
\begin{align*}
\theta_{19}(W(D,C))=
& \displaystyle \sum_{a,b\in \C(D,C);a<b} s_{a,b} \theta_{19} \Big( \big((a,a),(a,b)\big) \Big)\\
&+ t_1 \theta_{19} \Big(\big((0,1),(0,2)\big)\Big)- t_1 \theta_{19} \Big(\big((0,1),(0,3)\big)\Big)\\
&+ t_1 \theta_{19} \Big(\big((0,2),(0,3)\big)\Big)+ t_1 \theta_{19} \Big(\big((1,2),(1,3)\big)\Big)\\
=&15t_1 - 10t_1 + 4t_1 + 10t_1\\
=& 0,
\end{align*}
which contradicts $0 \not\in \Phi_{\theta_{19}}^{\rm NT}(K)$.

For the case that $\C(D,C)=\{0,1,2,3,7,12\}$, the sum of the weights of all the nontrivially colored crossings is
\begin{align*}
&&W(D,C) =& 
\displaystyle \sum_{a,b\in \C(D,C);a<b} s_{a,b}  \big((a,a),(a,b)\big) + t_1  \big((0,1),(0,2)\big) \\
&&&  + t_2  \big((0,1),(0,3)\big) + t_3  \big((0,2),(0,3)\big) + t_4  \big((0,7),(0,0)\big)\\
&&&   + t_5  \big((1,2),(1,3)\big) + t_6  \big((2,7),(2,12)\big)   + t_7  \big((2,12),(2,3)\big) 
\end{align*}
for some $s_{a,b}$'s, $t_1, \ldots , t_7 \in \mathbb Z$, where we note that when we regard $W(D,C)$ as an element of $C_2^{\rm SLB}(\mathbb Z_{19})$,  $s_{a,b}$'s should be dealt as elements of $\mathbb Z_2$.
Since
\begin{align*}
0=&\partial_2^{\rm SLB}(W(D,C))\\
=&\displaystyle \sum_{a\in \C(D,C)}   S_{a}   (a,a)  +(t_1+t_2)  (0,1)   +  (-t_1+t_3)  (0,2)  +(-t_2-t_3)   (0,3)  \\
& +t_4   (0,7)   -t_4  (0,12)  +(t_1+t_2-t_3+t_5)   (1,2)   +(t_3-t_5)   (1,3) \\
&  +(t_2+t_5-t_7)  (2,3)   +t_6  (2,7)  + (-t_6+t_7)   (2,12)  -t_7 (3,12) +(t_4+t_6)   (7,12)   
\end{align*}
holds for some $S_a \in \mathbb{Z}_2$ from Lemma~\ref{lem:weightsum2}, we obtain
\begin{align*}
&t_1=-t_2=t_3=t_5 \mbox{ and } t_4=t_6=t_7=0.
\end{align*}
Therefore, we have 
\begin{align*}
\theta_{19}(W(D,C))=
& \displaystyle \sum_{a,b\in \C(D,C);a<b} s_{a,b} \theta_{19} \Big( \big((a,a),(a,b)\big) \Big)\\
&+ t_1 \theta_{19} \Big(\big((0,1),(0,2)\big)\Big)- t_1 \theta_{19} \Big(\big((0,1),(0,3)\big)\Big)\\
&+ t_1 \theta_{19} \Big(\big((0,2),(0,3)\big)\Big)+ t_1 \theta_{19} \Big(\big((1,2),(1,3)\big)\Big)\\
=&15t_1 - 10t_1 + 4t_1 + 10t_1\\
=& 0,
\end{align*}
which contradicts $0 \not\in \Phi_{\theta_{19}}^{\rm NT}(K)$.

For the case that $\C(D,C)=\{0,1,2,4,5,10\}$, the sum of the weights of all the nontrivially colored crossings is
\begin{align*}
&&W(D,C) =& 
\displaystyle \sum_{a,b\in \C(D,C);a<b} s_{a,b}  \big((a,a),(a,b)\big) \\
&&&  + t_1  \big((0,1),(0,2)\big) + t_2  \big((0,1),(0,5)\big) + t_3  \big((0,2),(0,4)\big) \\
&&&  + t_4  \big((0,4),(0,5)\big) + t_5  \big((0,5),(0,10)\big) + t_6  \big((0,10),(0,1)\big) \\
&&&  + t_7  \big((1,2),(1,5)\big) + t_8  \big((1,4),(1,5)\big) 
\end{align*}
for some $s_{a,b}$'s, $t_1, \ldots , t_8 \in \mathbb Z$, where we note that when we regard $W(D,C)$ as an element of $C_2^{\rm SLB}(\mathbb Z_{19})$,  $s_{a,b}$'s should be dealt as elements of $\mathbb Z_2$.
Since
\begin{align*}
0=&\partial_2^{\rm SLB}(W(D,C))\\
=&\displaystyle \sum_{a\in \C(D,C)}   S_{a}   (a,a)  +(t_1+t_2-t_6)  (0,1)   +  (-t_1+t_3)  (0,2)  \\
&+(-t_3+t_4)   (0,4)   +(-t_2-t_4+t_5)   (0,5)   +(-t_5+t_6)  (0,10)  \\
&+(t_1+t_7)   (1,2)   +(t_2-t_4+t_8)   (1,4)   +(t_4-t_7-t_8)  (1,5)   \\
&-t_6  (1,10)  + (t_3+t_7-t_8)   (2,4)  +t_8   (2,5)   +(t_2+t_7)  (4,5)  + t_5   (5,10)  
\end{align*}
holds for some $S_a \in \mathbb{Z}_2$ from Lemma~\ref{lem:weightsum2}, we obtain
\begin{align*}
&t_1=t_2=t_3=t_4=t_5=t_6=t_7=t_8=0. 
\end{align*}
Therefore, we have 
\begin{align*}
\theta_{19}(W(D,C))=
& \displaystyle \sum_{a,b\in \C(D,C);a<b} s_{a,b} \theta_{19} \Big( \big((a,a),(a,b)\big) \Big)\\
=& 0,
\end{align*}
which contradicts $0 \not\in \Phi_{\theta_{19}}^{\rm NT}(K)$.

For the case that $\C(D,C)=\{0,1,2,4,5,14\}$, the sum of the weights of all the nontrivially colored crossings is
\begin{align*}
&&W(D,C) =& 
\displaystyle \sum_{a,b\in \C(D,C);a<b} s_{a,b}  \big((a,a),(a,b)\big) \\
&&&  + t_1  \big((0,1),(0,2)\big) + t_2  \big((0,1),(0,5)\big) + t_3  \big((0,2),(0,4)\big) \\
&&&  + t_4  \big((0,4),(0,5)\big) + t_5  \big((0,5),(0,0)\big) + t_6  \big((1,2),(1,5)\big) \\
&&&  + t_7  \big((1,4),(1,5)\big) + t_8  \big((4,14),(4,5)\big) 
\end{align*}
for some $s_{a,b}$'s, $t_1, \ldots , t_8 \in \mathbb Z$, where we note that when we regard $W(D,C)$ as an element of $C_2^{\rm SLB}(\mathbb Z_{19})$,  $s_{a,b}$'s should be dealt as elements of $\mathbb Z_2$.
Since
\begin{align*}
0=&\partial_2^{\rm SLB}(W(D,C))\\
=&\displaystyle \sum_{a\in \C(D,C)}   S_{a}   (a,a)  +(t_1+t_2)  (0,1)   +  (-t_1+t_3)  (0,2)  +(-t_3+t_4)   (0,4)  \\
& +(-t_2-t_4+t_5)   (0,5)   -t_5  (0,14)  +(t_1+t_6)   (1,2)  \\
& +(t_2-t_4+t_7)   (1,4)   +(t_4-t_6-t_7)  (1,5)  +(t_3+t_6-t_7)  (2,4)  \\
& + t_7   (2,5)  +(t_2+t_6-t_8)   (4,5)  +t_8  (4,14)  + (t_5-t_8)   (5,14)   
\end{align*}
holds for some $S_a \in \mathbb{Z}_2$ from Lemma~\ref{lem:weightsum2}, we obtain
\begin{align*}
&t_1=t_2=t_3=t_4=t_5=t_6=t_7=t_8=0. 
\end{align*}
Therefore, we have 
\begin{align*}
\theta_{19}(W(D,C))=
& \displaystyle \sum_{a,b\in \C(D,C);a<b} s_{a,b} \theta_{19} \Big( \big((a,a),(a,b)\big) \Big)\\
=& 0,
\end{align*}
which contradicts $0 \not\in \Phi_{\theta_{19}}^{\rm NT}(K)$. 

For the case that $\C(D,C)=\{0,1,2,4,7,12\}$, the sum of the weights of all the nontrivially colored crossings is
\begin{align*}
&&W(D,C) =& 
\displaystyle \sum_{a,b\in \C(D,C);a<b} s_{a,b}  \big((a,a),(a,b)\big) \\
&&&  + t_1  \big((0,1),(0,2)\big) + t_2  \big((0,2),(0,4)\big) + t_3  \big((0,7),(0,0)\big) \\
&&&  + t_4  \big((1,4),(1,7)\big) + t_5  \big((1,12),(1,4)\big) + t_6  \big((2,7),(2,12)\big) 
\end{align*}
for some $s_{a,b}$'s, $t_1, \ldots , t_6 \in \mathbb Z$, where we note that when we regard $W(D,C)$ as an element of $C_2^{\rm SLB}(\mathbb Z_{19})$,  $s_{a,b}$'s should be dealt as elements of $\mathbb Z_2$.
Since
\begin{align*}
0=&\partial_2^{\rm SLB}(W(D,C))\\
=&\displaystyle \sum_{a\in \C(D,C)}   S_{a}   (a,a)  +t_1  (0,1)   +  (-t_1+t_2)  (0,2) -t_2   (0,4)   +t_3   (0,7)  \\
&  -t_3  (0,12)  +t_1  (1,2)  +(t_4-t_5)   (1,4)   -t_4   (1,7)   +t_5  (1,12)  + t_2   (2,4)   \\
& +t_6   (2,7)   -t_6  (2,12)  + t_4   (4,7)  -t_5   (4,12)  +(t_3+t_6)  (7,12)  
\end{align*}
holds for some $S_a \in \mathbb{Z}_2$ from Lemma~\ref{lem:weightsum2}, we obtain
\begin{align*}
&t_1=t_2=t_3=t_4=t_5=t_6=0.
\end{align*}
Therefore, we have 
\begin{align*}
\theta_{19}(W(D,C))=
& \displaystyle \sum_{a,b\in \C(D,C);a<b} s_{a,b} \theta_{19} \Big( \big((a,a),(a,b)\big) \Big)\\
=& 0,
\end{align*}
which contradicts $0 \not\in \Phi_{\theta_{19}}^{\rm NT}(K)$.

For the case that $\C(D,C)=\{0,1,2,4,7,15\}$, the sum of the weights of all the nontrivially colored crossings is
\begin{align*}
&&W(D,C) =& 
\displaystyle \sum_{a,b\in \C(D,C);a<b} s_{a,b}  \big((a,a),(a,b)\big)+ t_1  \big((0,1),(0,2)\big) \\
&&&   + t_2  \big((0,2),(0,4)\big) + t_3  \big((0,4),(0,0)\big) + t_4  \big((1,4),(1,7)\big)\\
&&&   + t_5  \big((1,7),(1,2)\big) + t_6  \big((1,15),(1,2)\big)   + t_7  \big((4,15),(4,7)\big) 
\end{align*}
for some $s_{a,b}$'s, $t_1, \ldots , t_7 \in \mathbb Z$, where we note that when we regard $W(D,C)$ as an element of $C_2^{\rm SLB}(\mathbb Z_{19})$,  $s_{a,b}$'s should be dealt as elements of $\mathbb Z_2$.
Since
\begin{align*}
0=&\partial_2^{\rm SLB}(W(D,C))\\
=&\displaystyle \sum_{a\in \C(D,C)}   S_{a}   (a,a)  +t_1  (0,1)   +  (-t_1+t_2)  (0,2) +(-t_2+t_3)   (0,4)   -t_3   (0,15)  \\
&  +(t_1-t_5-t_6)  (1,2)  +t_4   (1,4)   +(-t_4+t_5)   (1,7)   +t_6  (1,15)   +t_2  (2,4)   \\
& -t_6   (2,7)-t_5   (2,15)   +(t_4-t_7)  (4,7)  + (t_3+t_7)   (4,15)  +(t_5-t_6-t_7)   (7,15)   
\end{align*}
holds for some $S_a \in \mathbb{Z}_2$ from Lemma~\ref{lem:weightsum2}, we obtain
\begin{align*}
&t_1=t_2=t_3=t_4=t_5=t_6=t_7=0.
\end{align*}
Therefore, we have 
\begin{align*}
\theta_{19}(W(D,C))=
& \displaystyle \sum_{a,b\in \C(D,C);a<b} s_{a,b} \theta_{19} \Big( \big((a,a),(a,b)\big) \Big)\\
=& 0,
\end{align*}
which contradicts $0 \not\in \Phi_{\theta_{19}}^{\rm NT}(K)$.

From the above discussion, we conclude that $\# \C(D,C) \geq 7$, and consequently  $\mincol_{19}(K)\geq 7$.
\end{proof}

\begin{proof}[Proof of Theorem~\ref{thm:main2}~{\rm (}$p=23${\rm )}]
Let $(D,C)$ be a nontrivially Dehn $23$-colored diagram of a knot $K$ with $0 \not\in \Phi_{\theta_{23}}^{\rm NT}(K)$.
Assume that $\#\C(D,C)=6$. 
Then, by Lemma~\ref{lemma:colorequiv} and Theorem~\ref{thm:colorcandidate}, we may assume that $\C(D,C)=\{0,1,2,3,6,12\}$, $\{0,1,2,4,7,12\}$, $\{0,1,2,4,7,13\}$, $\{0,1,2,4,7,14\}$, $\{0,1,2,4,9,14\}$ or $\{0,1,2,4,10,19\}$.

For the case that $\C(D,C)=\{0,1,2,3,6,12\}$, the sum of the weights of all the nontrivially colored crossings is
\begin{align*}
&&W(D,C) =& 
\displaystyle \sum_{a,b\in \C(D,C);a<b} s_{a,b}  \big((a,a),(a,b)\big) + t_1  \big((0,1),(0,2)\big)  \\
&&& + t_2  \big((0,1),(0,3)\big) + t_3  \big((0,2),(0,3)\big)+ t_4  \big((0,3),(0,6)\big) \\
&&&   + t_5  \big((0,6),(0,12)\big) + t_6  \big((0,12),(0,1)\big)  + t_7  \big((1,2),(1,3)\big) 
\end{align*}
for some $s_{a,b}$'s, $t_1, \ldots , t_7 \in \mathbb Z$, where we note that when we regard $W(D,C)$ as an element of $C_2^{\rm SLB}(\mathbb Z_{23})$,  $s_{a,b}$'s should be dealt as elements of $\mathbb Z_2$.
Since
\begin{align*}
0=&\partial_2^{\rm SLB}(W(D,C))\\
=&\displaystyle \sum_{a\in \C(D,C)}   S_{a}   (a,a)  +(t_1+t_2-t_6)  (0,1)   +  (-t_1+t_3)  (0,2)  +(-t_2-t_3+t_4)   (0,3)  \\
& +(-t_4+t_5)   (0,6)   +(-t_5+t_6)  (0,12)  +(t_1+t_2-t_3+t_7)   (1,2)  \\
& +(t_3-t_7)   (1,3)   -t_6  (1,12)  +(t_2+t_7)   (2,3)   + t_4   (3,6)   +t_5   (6,12)  
\end{align*}
holds for some $S_a \in \mathbb{Z}_2$ from Lemma~\ref{lem:weightsum2}, we obtain
\begin{align*}
&t_1=-t_2=t_3=t_7, \mbox{ and }t_4=t_5=t_6=0.  
\end{align*}
Therefore, we have 
\begin{align*}
\theta_{23}(W(D,C))=
& \displaystyle \sum_{a,b\in \C(D,C);a<b} s_{a,b} \theta_{23} \Big( \big((a,a),(a,b)\big) \Big)\\
&+ t_1 \theta_{19} \Big(\big((0,1),(0,2)\big)\Big)- t_1 \theta_{19} \Big(\big((0,1),(0,3)\big)\Big)\\
&+ t_1 \theta_{19} \Big(\big((0,2),(0,3)\big)\Big)+ t_1 \theta_{19} \Big(\big((1,2),(1,3)\big)\Big)\\
=&17t_1 - 9t_1 + 2t_1 + 13t_1\\
=& 0,
\end{align*}
which contradicts $0 \not\in \Phi_{\theta_{23}}^{\rm NT}(K)$.

For the case that $\C(D,C)=\{0,1,2,4,7,12\}$, the sum of the weights of all the nontrivially colored crossings is
\begin{align*}
&&W(D,C) =& 
\displaystyle \sum_{a,b\in \C(D,C);a<b} s_{a,b}  \big((a,a),(a,b)\big) \\
&&&  + t_1  \big((0,1),(0,2)\big) + t_2  \big((0,2),(0,4)\big) + t_3  \big((0,12),(0,1)\big) \\
&&&  + t_4  \big((1,4),(1,7)\big) + t_5  \big((2,7),(2,12)\big)  
\end{align*}
for some $s_{a,b}$'s, $t_1, \ldots , t_5 \in \mathbb Z$, where we note that when we regard $W(D,C)$ as an element of $C_2^{\rm SLB}(\mathbb Z_{23})$,  $s_{a,b}$'s should be dealt as elements of $\mathbb Z_2$.
Since
\begin{align*}
0=&\partial_2^{\rm SLB}(W(D,C))\\
=&\displaystyle \sum_{a\in \C(D,C)}   S_{a}   (a,a)  + (t_1-t_3)  (0,1)   +  (-t_1+t_2)  (0,2)-t_2   (0,4)   \\
&  +t_3   (0,12)   +t_1  (1,2)  +t_4   (1,4)   -t_4   (1,7)   - t_3  (1,12)  \\
&+t_2   (2,4)   + t_5   (2,7)   -t_5   (2,12) +t_4   (4,7)   + t_5   (7,12)   
\end{align*}
holds for some $S_a \in \mathbb{Z}_2$ from Lemma~\ref{lem:weightsum2}, we obtain
\begin{align*}
&t_1=t_2=t_3=t_4=t_5=0. 
\end{align*}
Therefore, we have 
\begin{align*}
\theta_{23}(W(D,C))=
& \displaystyle \sum_{a,b\in \C(D,C);a<b} s_{a,b} \theta_{23} \Big( \big((a,a),(a,b)\big) \Big)\\
=& 0,
\end{align*}
which contradicts $0 \not\in \Phi_{\theta_{23}}^{\rm NT}(K)$.

For the case that $\C(D,C)=\{0,1,2,4,7,13\}$, the sum of the weights of all the nontrivially colored crossings is
\begin{align*}
&&W(D,C) =& 
\displaystyle \sum_{a,b\in \C(D,C);a<b} s_{a,b}  \big((a,a),(a,b)\big) \\
&&&  + t_1  \big((0,1),(0,2)\big) + t_2  \big((0,2),(0,4)\big) + t_3  \big((1,4),(1,7)\big) \\
&&&  + t_4  \big((1,7),(1,13)\big) + t_5  \big((1,13),(1,2)\big)  
\end{align*}
for some $s_{a,b}$'s, $t_1, \ldots , t_5 \in \mathbb Z$, where we note that when we regard $W(D,C)$ as an element of $C_2^{\rm SLB}(\mathbb Z_{23})$,  $s_{a,b}$'s should be dealt as elements of $\mathbb Z_2$.
Since
\begin{align*}
0=&\partial_2^{\rm SLB}(W(D,C))\\
=&\displaystyle \sum_{a\in \C(D,C)}   S_{a}   (a,a)  + t_1  (0,1)   +  (-t_1+t_2)  (0,2)  -t_2   (0,4) \\
&  +(t_1-t_5)   (1,2)   +t_3  (1,4) +(-t_3+t_4)   (1,7)   +(-t_4+t_5)   (1,13) \\
&  + t_2  (2,4)  -t_5   (2,13)   + t_3   (4,7)   +t_4   (7,13)  
\end{align*}
holds for some $S_a \in \mathbb{Z}_2$ from Lemma~\ref{lem:weightsum2}, we obtain
\begin{align*}
&t_1=t_2=t_3=t_4=t_5=0. 
\end{align*}
Therefore, we have 
\begin{align*}
\theta_{23}(W(D,C))=
& \displaystyle \sum_{a,b\in \C(D,C);a<b} s_{a,b} \theta_{23} \Big( \big((a,a),(a,b)\big) \Big)\\
=& 0,
\end{align*}
which contradicts $0 \not\in \Phi_{\theta_{23}}^{\rm NT}(K)$.

For the case that $\C(D,C)=\{0,1,2,4,7,14\}$, the sum of the weights of all the nontrivially colored crossings is
\begin{align*}
&&W(D,C) =& 
\displaystyle \sum_{a,b\in \C(D,C);a<b} s_{a,b}  \big((a,a),(a,b)\big) \\
&&&  + t_1  \big((0,1),(0,2)\big) + t_2  \big((0,2),(0,4)\big) + t_3  \big((0,7),(0,14)\big) \\
&&&  + t_4  \big((1,4),(1,7)\big) + t_5  \big((1,14),(1,4)\big)  
\end{align*}
for some $s_{a,b}$'s, $t_1, \ldots , t_5 \in \mathbb Z$, where we note that when we regard $W(D,C)$ as an element of $C_2^{\rm SLB}(\mathbb Z_{23})$,  $s_{a,b}$'s should be dealt as elements of $\mathbb Z_2$.
Since
\begin{align*}
0=&\partial_2^{\rm SLB}(W(D,C))\\
=&\displaystyle \sum_{a\in \C(D,C)}   S_{a}   (a,a)  + t_1  (0,1)   +  (-t_1+t_2)  (0,2) -t_2   (0,4)  \\
&  +t_3   (0,7)   -t_3  (0,14) +t_1   (1,2)   +(t_4-t_5)   (1,4)   - t_4  (1,7)  \\
&+t_5   (1,14)   + t_2   (2,4)   +t_4   (4,7)  -t_5   (4,14)   + t_3   (7,14)   
\end{align*}
holds for some $S_a \in \mathbb{Z}_2$ from Lemma~\ref{lem:weightsum2}, we obtain
\begin{align*}
&t_1=t_2=t_3=t_4=t_5=0. 
\end{align*}
Therefore, we have 
\begin{align*}
\theta_{23}(W(D,C))=
& \displaystyle \sum_{a,b\in \C(D,C);a<b} s_{a,b} \theta_{23} \Big( \big((a,a),(a,b)\big) \Big)\\
=& 0,
\end{align*}
which contradicts $0 \not\in \Phi_{\theta_{23}}^{\rm NT}(K)$.

For the case that $\C(D,C)=\{0,1,2,4,9,14\}$, the sum of the weights of all the nontrivially colored crossings is
\begin{align*}
&&W(D,C) =& 
\displaystyle \sum_{a,b\in \C(D,C);a<b} s_{a,b}  \big((a,a),(a,b)\big) \\
&&&  + t_1  \big((0,1),(0,2)\big) + t_2  \big((0,2),(0,4)\big) + t_3  \big((0,9),(0,0)\big) \\
&&&  + t_4  \big((1,14),(1,4)\big) + t_5 \big((4,9),(4,14)\big)   
\end{align*}
for some $s_{a,b}$'s, $t_1, \ldots , t_5 \in \mathbb Z$, where we note that when we regard $W(D,C)$ as an element of $C_2^{\rm SLB}(\mathbb Z_{23})$,  $s_{a,b}$'s should be dealt as elements of $\mathbb Z_2$.
Since
\begin{align*}
0=&\partial_2^{\rm SLB}(W(D,C))\\
=&\displaystyle \sum_{a\in \C(D,C)}   S_{a}   (a,a)  + t_1  (0,1)   +  (-t_1+t_2)  (0,2) -t_2   (0,4)  \\
&  +t_3   (0,9)   -t_3  (0,14) +t_1   (1,2)   -t_4   (1,4)   + t_4  (1,14)  \\
&+t_2   (2,4)   + t_5   (4,9)   +(-t_4-t_5)   (4,14)  +(t_3+t_5)   (9,14)     
\end{align*}
holds for some $S_a \in \mathbb{Z}_2$ from Lemma~\ref{lem:weightsum2}, we obtain
\begin{align*}
&t_1=t_2=t_3=t_4=t_5=0. 
\end{align*}
Therefore, we have 
\begin{align*}
\theta_{23}(W(D,C))=
& \displaystyle \sum_{a,b\in \C(D,C);a<b} s_{a,b} \theta_{23} \Big( \big((a,a),(a,b)\big) \Big)\\
=& 0,
\end{align*}
which contradicts $0 \not\in \Phi_{\theta_{23}}^{\rm NT}(K)$.

For the case that $\C(D,C)=\{0,1,2,4,10,19\}$, the sum of the weights of all the nontrivially colored crossings is
\begin{align*}
&&W(D,C) =& 
\displaystyle \sum_{a,b\in \C(D,C);a<b} s_{a,b}  \big( (a,a),(a,b)\big)  \\
&&&  + t_1   \big((0,1),(0,2)\big)  + t_2  \big( (0,2),(0,4) \big) + t_3   \big((0,4),(0,0)\big)  \\
&&&  +t_4 \big(  (1,10),(1,19)\big) + t_5 \big(  (2,10),(2,4)\big)  + t_6 \big(  (2,19),(2,4)\big)   
\end{align*}
for some $s_{a,b}$'s, $t_1, \ldots , t_6 \in \mathbb Z$, where we note that when we regard $W(D,C)$ as an element of $C_2^{\rm SLB}(\mathbb Z_{23})$,  $s_{a,b}$'s should be dealt as elements of $\mathbb Z_2$.
Since
\begin{align*}
0=&\partial_2^{\rm SLB}(W(D,C))\\
=&\displaystyle \sum_{a\in \C(D,C)}   S_{a}   (a,a)  + t_1  (0,1)   +  (-t_1+t_2)  (0,2) +(-t_2+t_3)   (0,4)  \\
&  -t_3 (0,19)  +t_1   (1,2)   +t_4  (1,10) -t_4   (1,19)   +(t_2-t_5-t_6)   (2,4) \\
&   + t_5  (2,10)  +t_6   (2,19)   -t_6   (4,10)   +(t_3-t_5)   (4,19) +(t_4+t_5-t_6)   (10,19)     
\end{align*}
holds for some $S_a \in \mathbb{Z}_2$ from Lemma~\ref{lem:weightsum2}, we obtain
\begin{align*}
&t_1=t_2=t_3=t_4=t_5=t_6=0. 
\end{align*}
Therefore, we have 
\begin{align*}
\theta_{23}(W(D,C))=
& \displaystyle \sum_{a,b\in \C(D,C);a<b} s_{a,b} \theta_{23} \Big( \big((a,a),(a,b)\big) \Big)\\
=& 0,
\end{align*}
which contradicts $0 \not\in \Phi_{\theta_{23}}^{\rm NT}(K)$.


From the above discussion, we conclude that $\# \C(D,C) \geq 7$, and consequently  $\mincol_{23}(K)\geq 7$.
\end{proof}

\begin{proof}[Proof of Theorem~\ref{thm:main2}~{\rm (}$p=29${\rm )}]
When $p=29$, it is clear from Theorem~\ref{thm:main3} that $\mincol_{29}(K)\geq 7$ holds.
\end{proof}

\begin{proof}[Proof of Theorem~\ref{thm:main2}~{\rm (}$p=31${\rm )}]
Let $(D,C)$ be a nontrivially Dehn $31$-colored diagram of a knot $K$ with $0 \not\in \Phi_{\theta_{31}}^{\rm NT}(K)$.
Assume that $\#\C(D,C)=6$. 
Then, by Lemma~\ref{lemma:colorequiv} and Theorem~\ref{thm:colorcandidate}, we may assume that $\C(D,C)=\{0,1,2,4,8,16\}$.

For the case that $\C(D,C)=\{0,1,2,4,8,16\}$, the sum of the weights of all the nontrivially colored crossings is
\begin{align*}
&&W(D,C) =& 
\displaystyle \sum_{a,b\in \C(D,C);a<b} s_{a,b}  \big((a,a),(a,b)\big) \\
&&&  + t_1  \big((0,1),(0,2)\big) + t_2  \big((0,2),(0,4)\big) + t_3  \big((0,4),(0,8)\big) \\
&&&  + t_4  \big((0,8),(0,16)\big) + t_5  \big((0,16),(0,1)\big)  
\end{align*}
for some $s_{a,b}$'s, $t_1, \ldots , t_5 \in \mathbb Z$, where we note that when we regard $W(D,C)$ as an element of $C_2^{\rm SLB}(\mathbb Z_{31})$,  $s_{a,b}$'s should be dealt as elements of $\mathbb Z_2$.
Since
\begin{align*}
0=&\partial_2^{\rm SLB}(W(D,C))\\
=&\displaystyle \sum_{a\in \C(D,C)}   S_{a}   (a,a)  +(t_1-t_5)  (0,1)   +  (-t_1+t_2)  (0,2) +(-t_2+t_3)   (0,4)   \\
& +(-t_3+t_4)   (0,8)   +(-t_4+t_5)  (0,16) + t_1   (1,2)  -t_5   (1,16)   + t_2  (2,4) \\
&+t_3   (4,8)   + t_4   (8,16)   
\end{align*}
holds for some $S_a \in \mathbb{Z}_2$ from Lemma~\ref{lem:weightsum2}, we obtain
\begin{align*}
&t_1=t_2=t_3=t_4=t_5=0. 
\end{align*}
Therefore, we have 
\begin{align*}
\theta_{31}(W(D,C))=
& \displaystyle \sum_{a,b\in \C(D,C);a<b} s_{a,b} \theta_{31} \Big( \big((a,a),(a,b)\big) \Big)\\
=& 0,
\end{align*}
which contradicts $0 \not\in \Phi_{\theta_{31}}^{\rm NT}(K)$. 


From the above discussion, we conclude that $\# \C(D,C) \geq 7$, and consequently  $\mincol_{31}(K)\geq 7$.
\end{proof}

\section{Proof of Proposition~\ref{thm:main4}}

\begin{proof}[Proof of Proposition~\ref{thm:main4}~{\rm (}$p=7${\rm )}]
Let $K$ be the $5_2$ knot in the Rolfsen's knot table.
Since $K$ has a nontrivially Dehn $7$-colored diagram $(D,C)$ with $\#\C(D,C)=5$ as depicted in the left of Figure~\ref{knot5-2}, $\mincol_7(K)\leq 5$ holds.
\begin{figure}[t]
  \begin{center}
    \includegraphics[clip,width=7cm]{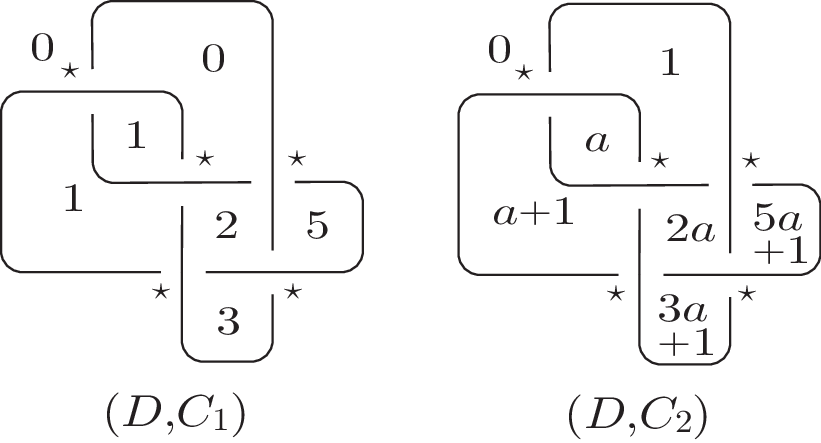}
    \caption{}
    \label{knot5-2}
  \end{center}
\end{figure}

For the nontrivially Dehn $7$-colored diagram $(D,C_1)$ of $K$ in the left of Figure~\ref{knot5-2}, since 
\begin{align*}
W(D,C_1) =& + \big((0,0),(0,1)\big) - \big((0,1),(0,2)\big)+ \big((0,5),(0,0)\big)\\
& +\big((0,1),(0,3)\big) +\big((0,3),(0,5)\big),
\end{align*}
we have 
\begin{align*}
\theta_7(W(D,C_1))= 
&+(0-0) \frac{  (0-0+2\cdot 1)^7 +( 0+0)^7- 2( 0+1)^7   }{7}\\
&-(0-1) \frac{  (0-1+2\cdot 2)^7 +( 0+1)^7- 2( 0+2)^7   }{7}\\
&+(0-5) \frac{  (0-5+2\cdot 0)^7 +( 0+5)^7- 2( 0+0)^7   }{7}\\
&+(0-1) \frac{  (0-1+2\cdot 3)^7 +( 0+1)^7- 2( 0+3)^7   }{7}\\
&+(0-3) \frac{  (0-3+2\cdot 5)^7 +( 0+3)^7- 2( 0+5)^7   }{7}\\
=&+0-4+0+6+3= 5 \not = 0\in \mathbb Z_7.
\end{align*}

For the nontrivially Dehn $7$-colored diagram $(D,C_2)$, with some $a\in \mathbb Z_7^{\times}$, of $K$ in the right of Figure~\ref{knot5-2}, since 
\begin{align*}
W(D,C_2) =& + \big((0,1),(0,a+1)\big) - \big((1,a),(1,2a)\big)+ \big((0,5a+1),(0,1)\big)\\
& +\big((0,a+1),(0,3a+1)\big) +\big((0,3a+1),(0,5a+1)\big),
\end{align*}
we have 
\begin{align*}
&\theta_7(W(D,C_2))\\
&=+(0-1) \frac{  (0-1+2\cdot (a+1))^7 +( 0+1)^7- 2( 0+(a+1))^7}{7}\\
&\hspace{1em}-(1-a) \frac{  (1-a+2\cdot (2a))^7 +( 1+a)^7- 2( 1+(2a))^7}{7}\\
&\hspace{1em}+(0-(5a+1)) \frac{  (0-(5a+1)+2\cdot 1)^7 +( 0+(5a+1))^7- 2( 0+1)^7}{7}\\
&\hspace{1em}+(0-(a+1)) \frac{  (0-(a+1)+2\cdot (3a+1))^7 +( 0+(a+1))^7- 2( 0+(3a+1))^7}{7}\\
&\hspace{1em}+(0-(3a+1)) \frac{  (0-(3a+1)+2\cdot (5a+1))^7 +( 0+(3a+1))^7- 2( 0+(5a+1))^7}{7}\\
&=+(1 + 3 a + a^2 + 5 a^3 + a^5)
-(6 + 3 a + 3 a^2 + 5 a^3 + a^4 + 3 a^5)\\
&\hspace{1em}+(5 + 4 a + 4 a^2 + 6 a^3 + a^4 + 5 a^5)
+(4 + 6 a + 3 a^2 + a^3 + a^4 + a^5)\\
&\hspace{1em}+(3 + 4 a + 6 a^4 + 3 a^5)\\
&=5a^2 \not = 0 \in \mathbb Z_7.
\end{align*}

For any nontrivial Dehn $7$-coloring $C$ of $D$, $C$ is affine equivalent to $C_1$ or $C_2$ in Figure~\ref{knot5-2}, and more precisely, $C\sim C_1$ if two colors are appeared around the upper left crossing on $(D,C)$, and  $C\sim C_2$ otherwise.
Since we have $sC+t = C_1$ or $C_2$ for some $s\in \mathbb Z_7^{\times}$ and $t\in \mathbb Z_7$, by Proposition~\ref{thm:cocycleinvariantforaffineequivalence},
$$\theta_7 (W(D,C))= s^2 \theta_7(W(D,C_1)) \mbox{ or } s^2 \theta_7(W(D,C_2)),$$
which implies that $\theta_7 (W(D,C))\not = 0$.
Therefore, $\mincol_7(K)\geq 5$ holds from Theorem~\ref{thm:main2}, and thus, we have  $\mincol_7(K)= 5$. 

We remark that 
\begin{align*}
&\Phi_{\theta_{7}}^{\rm NT}(K)\\
&=\{ 1\mbox{ }(49\mbox{ times}), 2\mbox{ }(49\mbox{ times}), 3\mbox{ }(49\mbox{ times}), 4\mbox{ }(49\mbox{ times}), 5\mbox{ }(49\mbox{ times}), 6\mbox{ }(49\mbox{ times}) \}.
\end{align*}
\end{proof}

Similar arguments apply to the other cases, and hence, we will briefly sketch the proofs of the cases when $p\in \{11,13, 19, 23, 29, 31\}$.

\begin{proof}[Proof of Proposition~\ref{thm:main4}~{\rm (}$p=11${\rm )}]
Let $K$ be the knot represented by the diagram in Figure~\ref{p=11diagex2}. First, we have a knot diagram of $K$ that is colored by $6$ colors (see Figure~\ref{p=11diagex2}), and thus $\mincol_{11}(K) \leq \lfloor \log_2 11 \rfloor +3 = 6$.

Moreover, we have
\begin{align*}
&\Phi_{\theta_{11}}^{\rm NT}(K)\\
&=\{ 2\mbox{ }(242\mbox{ times}), 6\mbox{ }(242\mbox{ times}), 7\mbox{ }(242\mbox{ times}), 8\mbox{ }(242\mbox{ times}), 10\mbox{ }(242\mbox{ times}) \}.
\end{align*}
Since $0 \notin \Phi_{\theta_{11}}^{\rm NT}(K)$, we have $\mincol_{11}(K) \geq \lfloor \log_2 11 \rfloor +3 = 6$ from Theorem~\ref{thm:main2}. 

Hence, we obtain  $\mincol_{11}(K) = \lfloor \log_2 11 \rfloor +3 = 6$. 

\begin{figure}[h]
\begin{center}
\includegraphics[width=3.4cm]{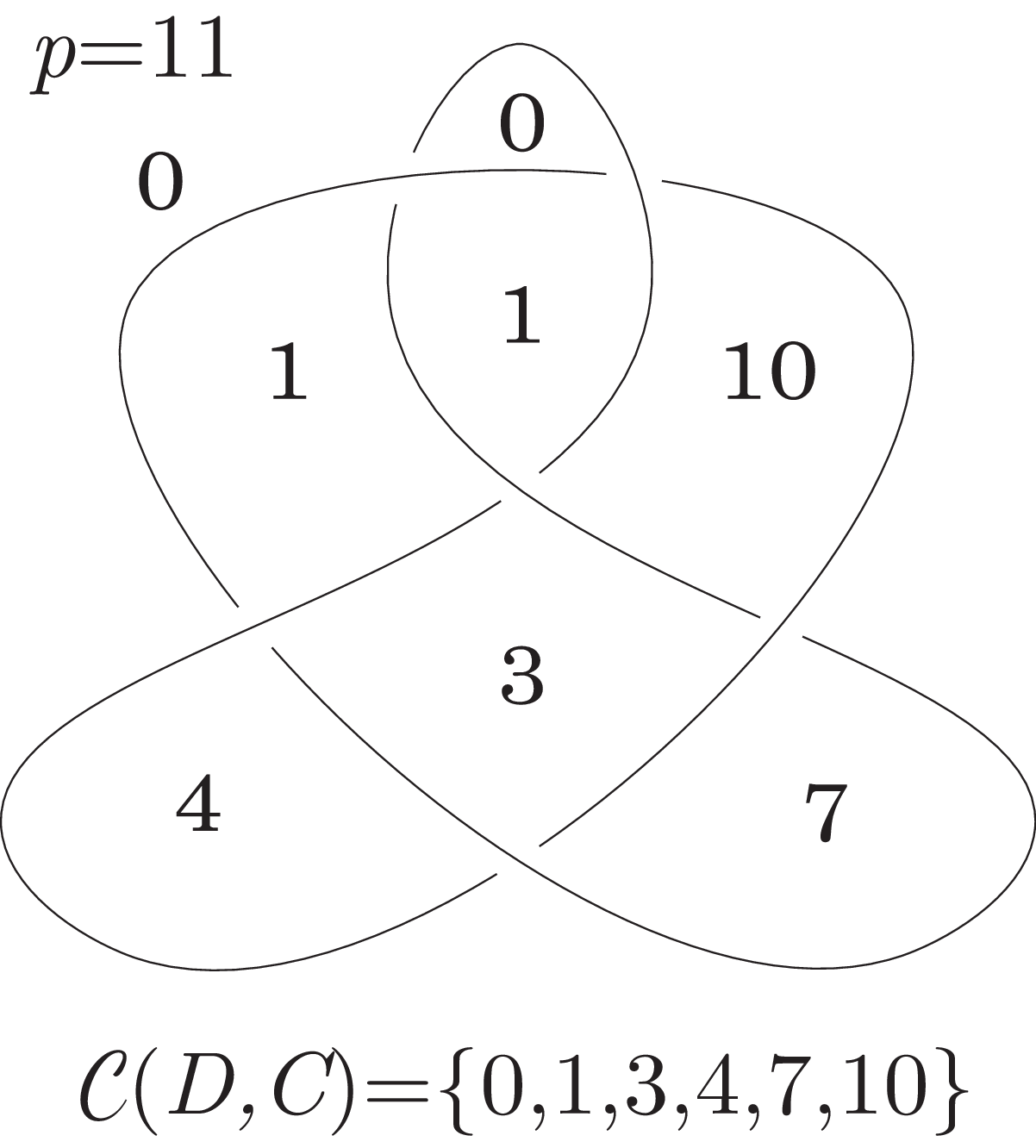}
\caption{} 
\label{p=11diagex2}
\end{center}
\end{figure}

\end{proof}

\begin{proof}[Proof of Proposition~\ref{thm:main4}~{\rm (}$p=13${\rm )}]
Let $K$ be the knot represented by the diagram in Figure~\ref{p=13diagex2}. First, we have a knot diagram of $K$ that is colored by $6$ colors (see Figure~\ref{p=13diagex2}), and thus $\mincol_{13}(K) \leq \lfloor \log_2 13 \rfloor +3 = 6$. 

Moreover, it is clear from Theorem~\ref{thm:main3} that $\mincol_{13}(K) \geq \lfloor \log_2 13 \rfloor +3 = 6$ holds.

Hence, we obtain  $\mincol_{13}(K) = \lfloor \log_2 13 \rfloor +3 = 6$. 

\begin{figure}[h]
\begin{center}
\includegraphics[width=5.3cm]{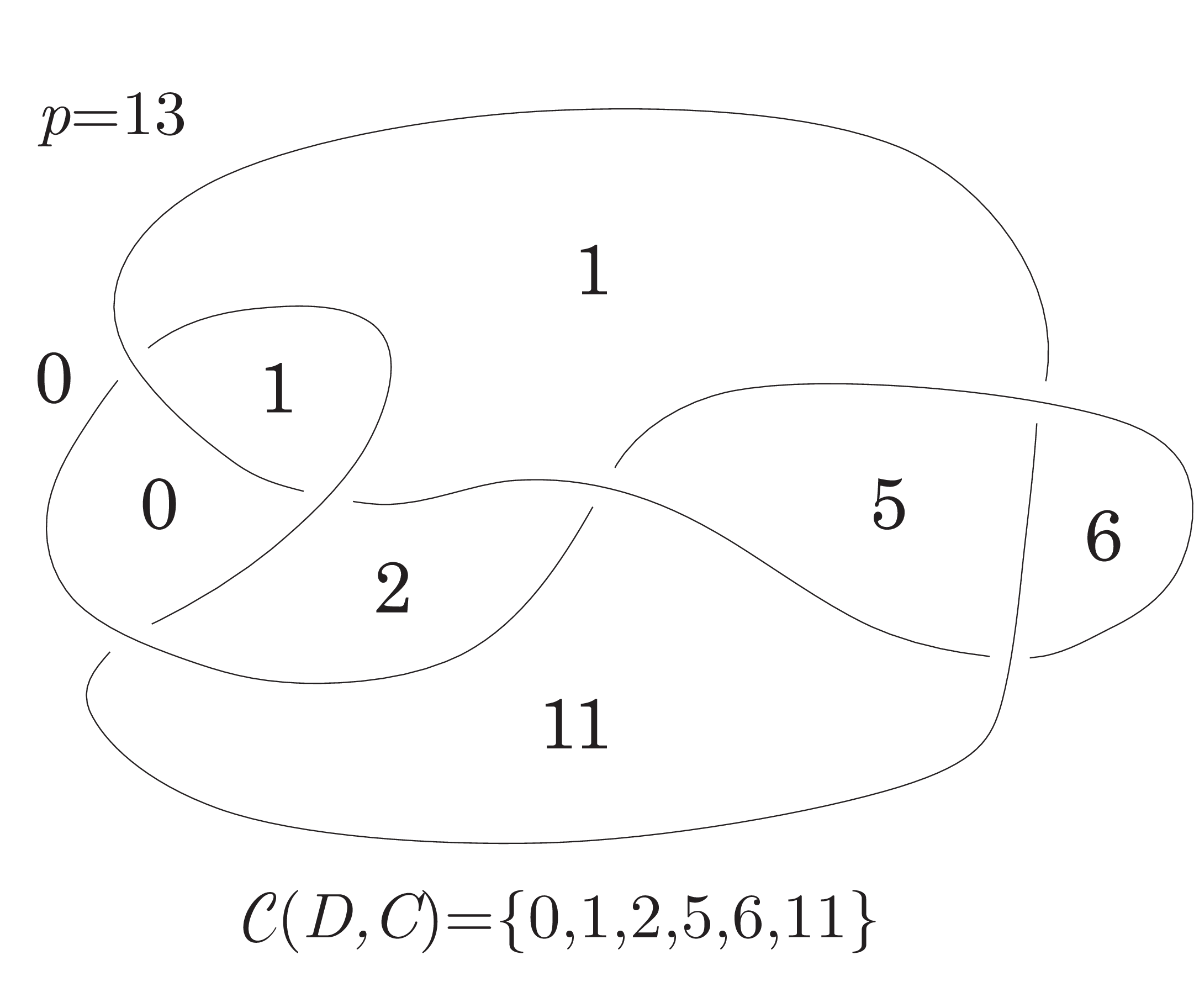}
\caption{} 
\label{p=13diagex2}
\end{center}
\end{figure}

\end{proof}

\begin{proof}[Proof of Proposition~\ref{thm:main4}~{\rm (}$p=19${\rm )}]
Let $K$ be the knot represented by the diagram in Figure~\ref{p=19diagex2}. First, we have a knot diagram of $K$ that is colored by $7$ colors (see Figure~\ref{p=11diagex2}), and thus $\mincol_{19}(K) \leq \lfloor \log_2 19 \rfloor +3 = 7$. 

Moreover, we have
\begin{align*}
&\Phi_{\theta_{19}}^{\rm NT}(K)\\
&= \{ 2\mbox{ }(722\mbox{ times}), 3\mbox{ }(722\mbox{ times}), 8\mbox{ }(722\mbox{ times}), 10\mbox{ }(722\mbox{ times}), \\
&\hspace{1em} 12\mbox{ }(722\mbox{ times}), 13\mbox{ }(722\mbox{ times}), 14\mbox{ }(722\mbox{ times}), 15\mbox{ }(722\mbox{ times}),\\ 
&\hspace{1em} 18\mbox{ }(722\mbox{ times}) \}.
\end{align*} 
Since $0 \notin \Phi_{\theta_{19}}^{\rm NT}(K)$, we have $\mincol_{19}(K) \geq \lfloor \log_2 19 \rfloor +3 = 7$ from Theorem~\ref{thm:main2}. 

Hence, we obtain  $\mincol_{19}(K) = \lfloor \log_2 19 \rfloor +3 = 7$. 

\begin{figure}[h]
\begin{center}
\includegraphics[width=4cm]{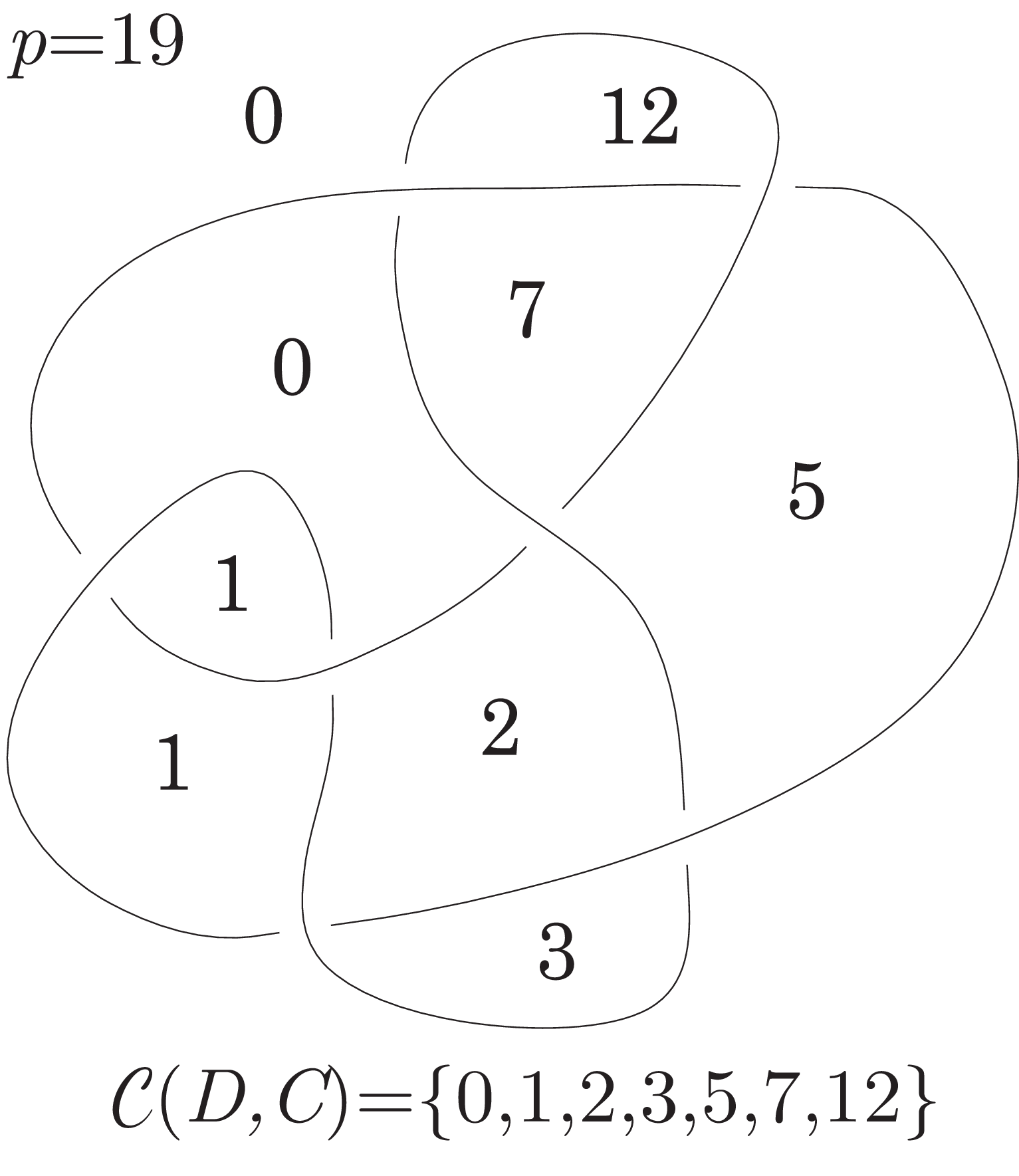}
\caption{} 
\label{p=19diagex2}
\end{center}
\end{figure}

\end{proof}

\begin{proof}[Proof of Proposition~\ref{thm:main4}~{\rm (}$p=23${\rm )}]
Let $K$ be the knot represented by the diagram in Figure~\ref{p=23diagex2}. First, we have a knot diagram of $K$ that is colored by $8$ colors (see Figure~\ref{p=23diagex2}), and thus $\mincol_{23}(K) \leq \lfloor \log_2 23 \rfloor +4 = 8$. 

Moreover, we have
\begin{align*}
&\Phi_{\theta_{23}}^{\rm NT}(K)\\
&= \{ 1\mbox{ }(1058\mbox{ times}), 2\mbox{ }(1058\mbox{ times}), 3\mbox{ }(1058\mbox{ times}), 4\mbox{ }(1058\mbox{ times}),\\
&\hspace{1em}6\mbox{ }(1058\mbox{ times}),8\mbox{ }(1058\mbox{ times}), 9\mbox{ }(1058\mbox{ times}), 12\mbox{ }(1058\mbox{ times}), \\
&\hspace{1em}13\mbox{ }(1058\mbox{ times}), 16\mbox{ }(1058\mbox{ times}), 18\mbox{ }(1058\mbox{ times}) \}. 
\end{align*}
Since $0 \notin \Phi_{\theta_{23}}^{\rm NT}(K)$, we have $\mincol_{23}(K) \geq \lfloor \log_2 23 \rfloor +3 = 7$ from Theorem~\ref{thm:main2}. 

Hence, we obtain  $\mincol_{23}(K) = \lfloor \log_2 23 \rfloor +3 \mbox{ or } \lfloor \log_2 23 \rfloor +4$. 

\begin{figure}[h]
\begin{center}
\includegraphics[width=5cm]{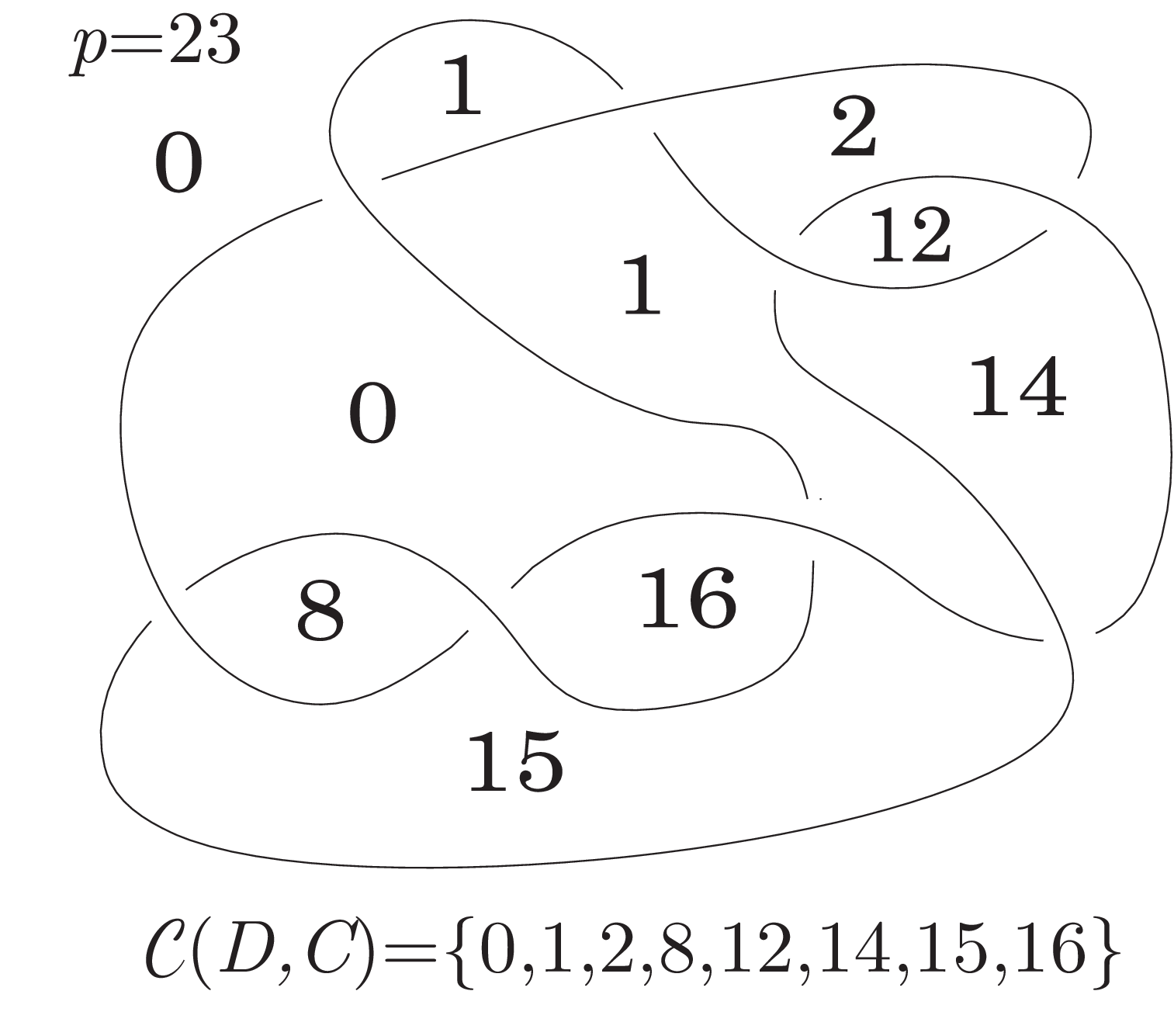}
\caption{} 
\label{p=23diagex2}
\end{center}
\end{figure}

\end{proof}

\begin{proof}[Proof of Proposition~\ref{thm:main4}~{\rm (}$p=29${\rm )}]
Let $K$ be the knot represented by the diagram in Figure~\ref{p=29diagex}. First, we have a knot diagram of $K$ that is colored by $7$ colors (see Figure~\ref{p=29diagex}), and thus $\mincol_{29}(K) \leq \lfloor \log_2 29 \rfloor +3 = 7$. 

Moreover, it is clear from Theorem~\ref{thm:main3} that $\mincol_{29}(K) \geq \lfloor \log_2 29 \rfloor +3 = 7$ holds.

Hence, we obtain  $\mincol_{29}(K) = \lfloor \log_2 29 \rfloor +3 = 7$. 

\begin{figure}[h]
\begin{center}
\includegraphics[width=6cm]{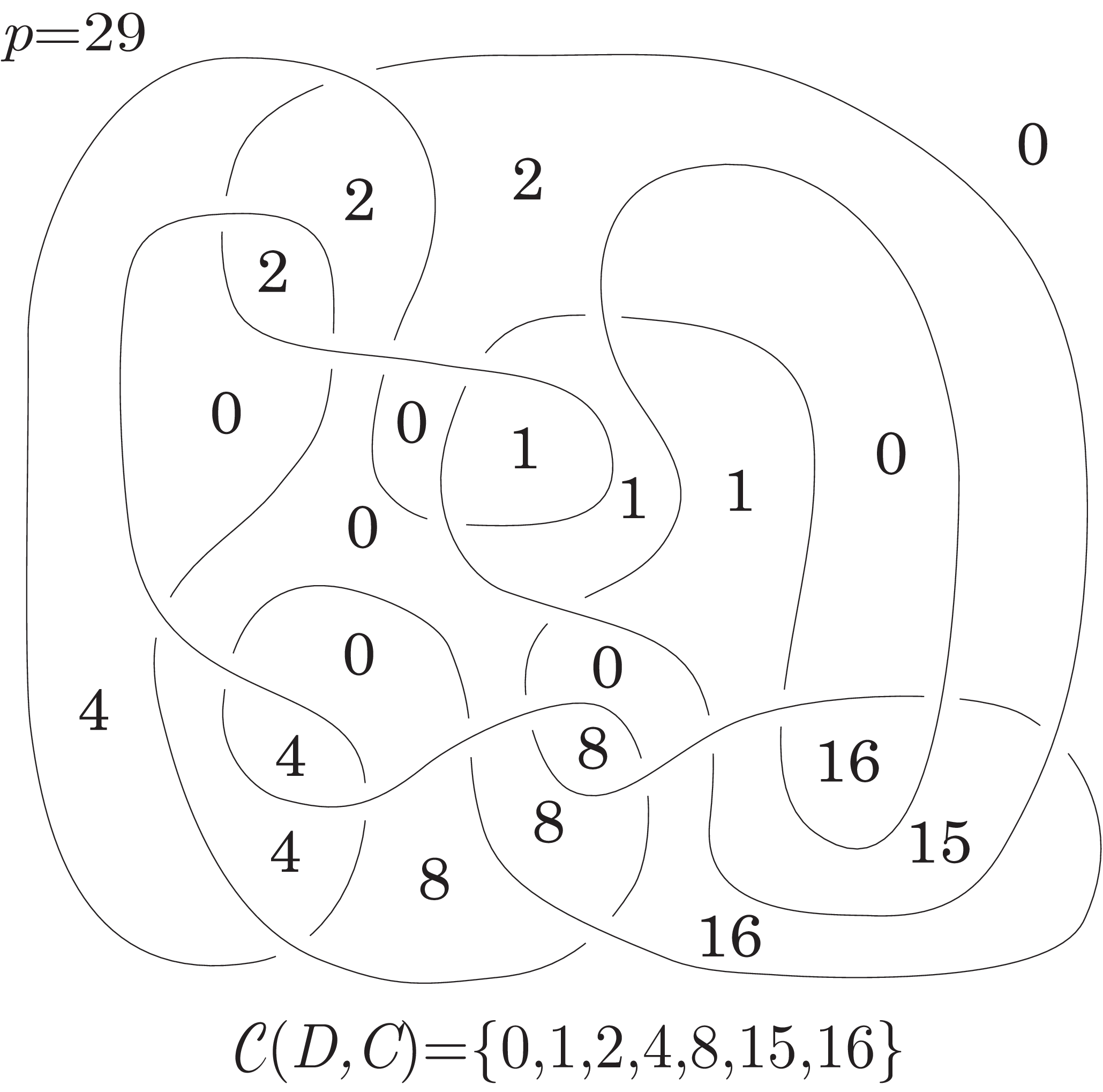}
\caption{} 
\label{p=29diagex}
\end{center}
\end{figure}

\end{proof}

\begin{proof}[Proof of Proposition~\ref{thm:main4}~{\rm (}$p=31${\rm )}]
Let $K$ be the knot represented by the diagram in Figure~\ref{p=31diagex}. First, we have a knot diagram of $K$ that is colored by $8$ colors (see Figure~\ref{p=31diagex}), and thus $\mincol_{31}(K) \leq \lfloor \log_2 31 \rfloor +4 = 8$. 

Moreover, we have
\begin{align*}
&\Phi_{\theta_{31}}^{\rm NT}(K)\\
&= \{ 1\mbox{ }(1922\mbox{ times}), 2\mbox{ }(1922\mbox{ times}), 4\mbox{ }(1922\mbox{ times}), 5\mbox{ }(1922\mbox{ times}),\\
&\hspace{1em}7\mbox{ }(1922\mbox{ times}), 8\mbox{ }(1922\mbox{ times}), 9\mbox{ }(1922\mbox{ times}), 10\mbox{ }(1922\mbox{ times}), \\
&\hspace{1em}14\mbox{ }(1922\mbox{ times}), 16\mbox{ }(1922\mbox{ times}), 18\mbox{ }(1922\mbox{ times}), 19\mbox{ }(1922\mbox{ times}), \\
&\hspace{1em}20\mbox{ }(1922\mbox{ times}), 25\mbox{ }(1922\mbox{ times}), 28\mbox{ }(1922\mbox{ times}) \}. 
\end{align*}
Since $0 \notin \Phi_{\theta_{31}}^{\rm NT}(K)$, we have $\mincol_{31}(K) \geq \lfloor \log_2 31 \rfloor +3 = 7$ from Theorem~\ref{thm:main2}. 

Hence, we obtain  $\mincol_{31}(K) = \lfloor \log_2 31 \rfloor +3 \mbox{ or } \lfloor \log_2 31 \rfloor +4$. 

\begin{figure}[h]
\begin{center}
\includegraphics[width=5cm]{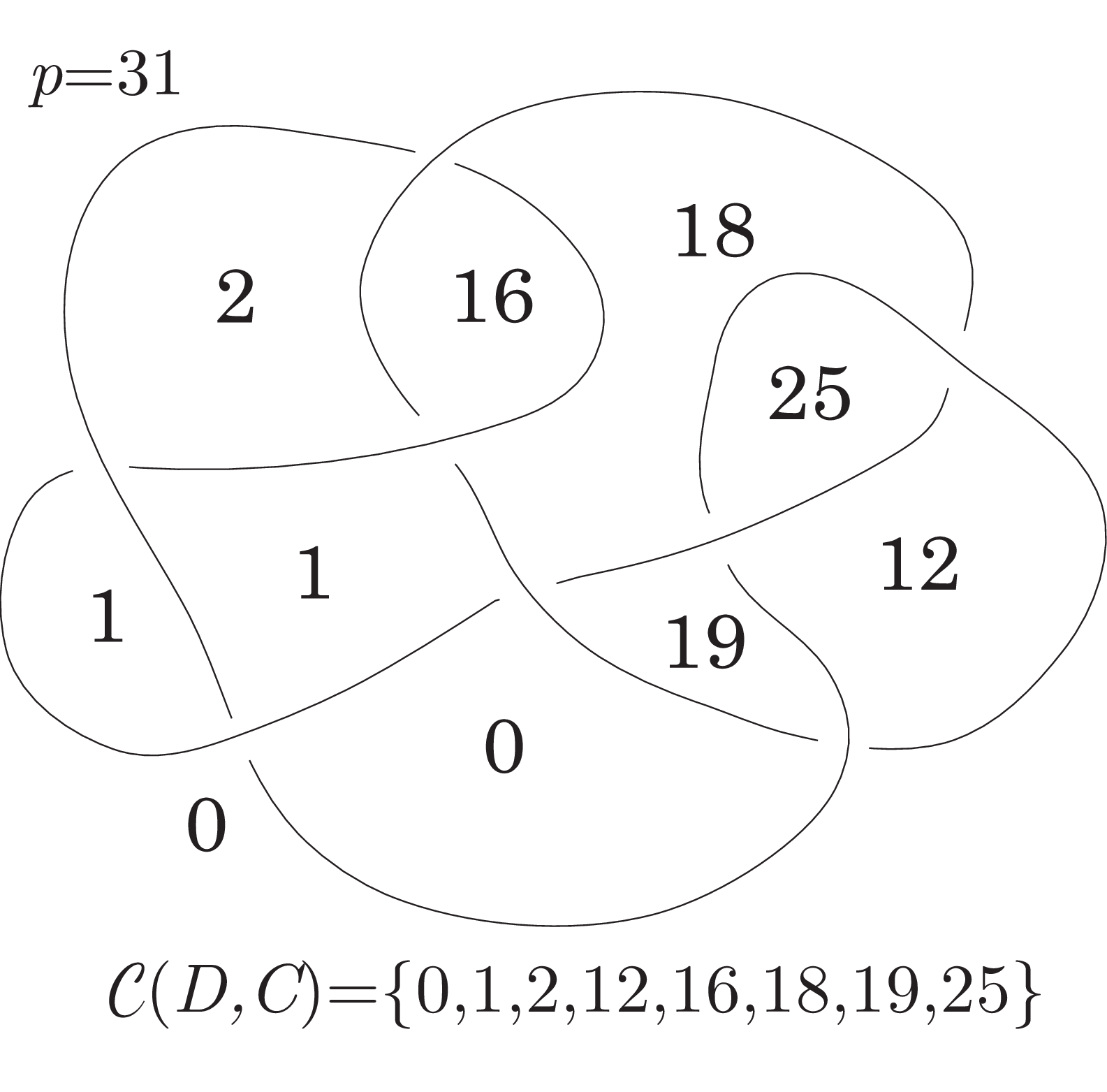}
\caption{} 
\label{p=31diagex}
\end{center}
\end{figure}

\end{proof}

\section*{Acknowledgments}

The second author was supported by JSPS KAKENHI Grant Number 21K03233.

\end{document}